\documentclass[a4paper,10pt]{article}
\usepackage[utf8]{inputenc}
\usepackage[T1]{fontenc}
\usepackage{mathtools}
\usepackage{amsmath}
\usepackage{amsthm}
\usepackage{amssymb}
\usepackage[retainorgcmds]{IEEEtrantools}
\usepackage{stackrel}
\usepackage[all]{xy}
\usepackage{setspace}

\usepackage{framed}
\usepackage{marginnote}
\usepackage{stmaryrd}
\usepackage{leftidx}

\theoremstyle{plain}
\newtheorem{thm}{Theorem}
\newtheorem{lem}{Lemma}
\newtheorem{pr}{Proposition}
\newtheorem{cor}{Corollary}
\theoremstyle{definition} 

\theoremstyle{remark} 
\newtheorem{rem}{Remark}
\newtheoremstyle{claim}
  {0.5\topsep}   
  {0.5\topsep}   
  {\itshape}  
  {0pt}       
  {\itshape} 
  {.}         
  {5pt plus 1pt minus 1pt} 
  {#1 #3}          
\theoremstyle{claim}

\let\AA\relax
\newcommand{\AA}{\mathbb{A}}
\newcommand{\afr}{\mathfrak{a}}
\newcommand{\bfr}{\mathfrak{b}}
\newcommand{\CC}{\mathbb{C}}
\newcommand{\EE}{\mathbb{E}}

\newcommand{\Gc}{\mathcal{G}}
\newcommand{\Ic}{\mathcal{I}}
\newcommand{\Kc}{\mathcal{K}}

\newcommand{\NN}{\mathbb{N}}
\newcommand{\OO}[1]{O\left(#1\right)}
\newcommand{\oo}[1]{o\left(#1\right)}
\newcommand{\QQ}{\mathbb{Q}}
\newcommand{\RR}{\mathbb{R}}
\newcommand{\TT}{\mathbb{T}}
\newcommand{\UU}{\mathbb{U}}
\newcommand{\xb}{\mathbf{x}}
\newcommand{\yb}{\mathbf{y}}
\newcommand{\ZZ}{\mathbb{Z}}
\newcommand{\Zc}{\mathcal{Z}}
\newcommand{\Pb}{\mathbb{P}}
\renewcommand{\Pr}[1]{\mathbb{P}\left[#1\right]}
\newcommand{\E}{\mathbb{E}}
\newcommand{\Me}[1]{\mathbb{E}\left[#1\right]}

\makeatletter
\newsavebox\myboxA
\newsavebox\myboxB
\newlength\mylenA
\newcommand*\xoverline[2][0.8]{
  \sbox{\myboxA}{$\m@th#2$}
  \setbox\myboxB\null
  \ht\myboxB=\ht\myboxA
  \dp\myboxB=\dp\myboxA
  \wd\myboxB=#1\wd\myboxA
  \sbox\myboxB{$\m@th\overline{\copy\myboxB}$}
  \setlength\mylenA{\the\wd\myboxA}
  \addtolength\mylenA{-\the\wd\myboxB}
  \ifdim\wd\myboxB<\wd\myboxA
     \rlap{\hskip 0.5\mylenA\usebox\myboxB}{\usebox\myboxA}
  \else
     \hskip -0.5\mylenA\rlap{\usebox\myboxA}{\hskip 0.5\mylenA\usebox\myboxB}
  \fi}
\makeatother

\newcommand{\xo}[1]{\xoverline{#1}}

\DeclareMathOperator{\tr}{{\mathrm{tr}}}

\let\Im\relax
\DeclareMathOperator{\Im}{\mathrm{Im}}

\let\Re\relax
\DeclareMathOperator{\Re}{\mathrm{Re}}

\newenvironment{EA}[1]{\begin{IEEEeqnarray*}{#1}}{\end{IEEEeqnarray*}}
\newcommand{\EAy}{\IEEEyesnumber}

\newcommand{\LRI}[4]{\,\leftidx{^{#3}}{#1}{^{#2}_{#4}}}
\newcommand{\Mes}[2]{(1-\mathbb{E}_{#1})\left[#2\right]}

\usepackage{hyperref}
\begin{document}
\title{No outliers in the spectrum of the product of independent non-Hermitian random matrices with independent entries}
\author{Yuriy Nemish\\
\\ \textit{University of Toulouse, France}}
\date{\today}
\maketitle
\begin{abstract}
  We consider products of independent square random non-Hermitian matrices.
More precisely, let $n\geq 2$ and let $X_1,\ldots,X_n$ be independent $N\times N$ random matrices with independent centered entries (either real or complex with independent real and imaginary parts) with variance $N^{-1}$.
In \cite{GotzTikh} and \cite{OrouSosh} it was shown that the limit of the empirical spectral distribution of the product $X_1\cdots X_n$ is supported in the unit disk.
We prove that if the entries of the matrices $X_1,\ldots,X_n$ satisfy uniform subexponential decay condition, then the spectral radius of $X_1\cdots X_n$ converges to 1 almost surely as $N\rightarrow \infty$.

\end{abstract}
\section{Introduction and main result}
\label{sec:introduction}

In this paper we consider the asymptotic behaviour of the spectral radius of a product of independent non-Hermitian random matrices. 
We start with a short overview of some recent results concerning products of non-Hermitian matrices.

One of the most studied models of non-Hermitian matrices is the Ginibre ensemble, in which the entries of the $N\times N$ matrix are centered i.i.d. Gaussian random variables with variance $N^{-1}$. 
Although first results concerning Ginibre matrices (\cite{Gini}) date back to 1965, many important results about the products of Ginibre matrices were obtained only quite recently. In \cite{BurdJaniWacl} Burda, Janik and Waclaw showed that the empirical spectral distribution (ESD) of a product of $n$ independent Ginibre matrices converges to the $n$th power of the uniform distribution on unit disk (circular law). 
In \cite{AkemBurd} Akemann and Burda studied the asymptotic behaviour of the $k$-point correlation functions for a product of independent Ginibre matrices.
The law of the eigenvalues of Ginibre matrix has a determinantal structure and due to this fact it is possible to compute the asymptotics of the correlation functions and to obtain many limiting properties of the distribution of the eigenvalues.
One interesting and important problem is to extend the above asymptotic properties of Ginibre matrices to a wider class of random matrices for which we do not have the determinantal structure, i.e. to show that these properties are universal.
The universality of the limiting ESD for products of independent non-hermitian matrices was shown by G\"{o}tze and Tikhomirov in \cite{GotzTikh}, and by  O'Rourke and Soshnikov  in \cite{OrouSosh}. They considered $N\times N$ matrices with independent \emph{identically distributed} centered entries of variance $N^{-1}$ (with an additional $(2+\epsilon)$th moment condition in \cite{OrouSosh}).
In our paper we show that for a large class of non-Hermitian random matrices with independent entries, for which the limiting empirical spectral measure is supported on the unit disk, the spectrum of the product of these matrices has no outliers, i.e. the spectral radius converges to 1.

We now define the model under consideration.
Fix an integer $n\geq 2$.
For each $a\in \llbracket 1,n\rrbracket$ and $N\in \NN$, let $X_a^N=(\LRI{x}{N}{a}{ij})_{1\leq i,j\leq N}$ be a matrix with independent centered entries (real or complex with independent real and imaginary parts) of variance $N^{-1}$.
The superscript $N$ will be mostly omitted.
We assume that the entries of the matrices satisfy the uniform subexponential decay condition, i.e. there exists $\theta>0$ independent of $N$ such that 
\begin{EA}{c}
  \EAy\label{eq:unif_decay_cond}
  \sup_{a\in\llbracket 1,n\rrbracket}
  \sup_{1\leq i,j\leq N}
  \Pr{|\sqrt{N}\LRI{x}{}{a}{ij}|>t}
  \leq
  \theta^{-1}e^{-t^{\theta}}
  .
\end{EA}

The main result of this work is the following statement.

\begin{thm}\label{thm:convg_spectral_radius}
  Let $X_1,\ldots,X_n$ be independent $N\times N$ matrices with independent centered entries (real or complex with independent real and imaginary parts) of variance $N^{-1}$ that satisfy the uniform subexponential decay condition \eqref{eq:unif_decay_cond}. 
 Then the spectral radius of the product $X_1\cdots X_n$ converges almost surely to 1 as $N\rightarrow \infty$.
\end{thm}

We comment briefly on the methods we used to prove Theorem~\ref{thm:convg_spectral_radius} and the structure of the  article. Firstly, we linearise our problem using the same trick as in \cite{AkemBurd}, \cite{BurdJaniWacl} and \cite{OrouSosh}. This allows us to study one $nN\times nN$ matrix $X$ with matrices $X_a, 1\leq a\leq n,$ as blocks, instead of the product $X_1\cdots X_n$. To this matrix we apply standard hermitization techniques and we use the approach of Bai and Silverstein to reduce the initial problem to the study of the Stieltjes transform of the matrices $(X-z)^*(X-z)$ around the origin for $|z|\geq 1+\delta >1$. All this is done in Section~\ref{sec:proofs1}. Due to the simple structure of the matrix $X$, it is possible to apply the machinery developed by Bourgade, Yau and Yin in \cite{BourYauYin} in order to obtain the concentration of the Stieltjes transform of the matrices $(X-z)^*(X-z)$ around the origin. 
In Sections \ref{sec:notation} and \ref{sec:tools} we introduce the notation and state some necessary preliminary results. 
In the last section we adjust the argument of Bourgade, Yau and Yin, so that we can use it for our model. 
In \cite{BourYauYin} a $N\times N$ matrix with independent centered entries of variance $N^{-1}$ was studied. 
However, there are many zero entries in the matrix $X$ that we consider, therefore some modifications are needed in order to get the concentration of the Stieltjes transform.
\begin{rem}
  The case of the product of matrices with independent \emph{identically distributed} entries can be efficiently studied using the moments method. 
For example, our result in the iid case can be derived following the proof in \cite{BaiYin}.
But it is still not clear for the author if it can be obtained using the moments method in the general case.
\end{rem}
\begin{rem}
  In two recent papers (\cite{AjanErdoKrug1} and \cite{AjanErdoKrug2}), Ajanki, Erd\"{o}s and Kr\"{u}ger studied the local properties of matrices of a very general type and were able to prove that for these models the local law holds  up to the optimal scale.
However, our hermitization matrix $(X-z)^*(X-z)$ does not fall into the class of matrices considered in these papers, therefore our result cannot be deduced directly from \cite{AjanErdoKrug1} and \cite{AjanErdoKrug2}.
\end{rem}
\begin{rem}
In a forthcoming paper we use the simple structure of the matrix $X$ to prove the concentration of the Stieltjes transform on the support of the limiting ESD of $(X-z)^*(X-z)$ and thus to get the local law for the products of non-Hermitian random matrices using basically the same techniques.
\end{rem}



\section{Structure of the proof of Theorem~\ref{thm:convg_spectral_radius}}
\label{sec:proofs1}
We show that for any $\delta>0$ almost surely, for $N$ sufficiently large, all the eigenvalues of the product $X_1\cdots X_n$ are contained in the disk of radius $1+\delta$. Theorem~\ref{thm:convg_spectral_radius} then follows easily.

Define the matrix
\begin{equation}\label{eq:linearisation_matrix}
  X=
  \begin{pmatrix}
    0 & X_1 & 0 &  \cdots & 0 \\
    0 & 0 & X_2 &  \cdots & 0 \\
     & & &\ddots &  \\
    0 & 0 & 0 & \cdots & X_{n-1}\\
    X_n& 0  & 0 &\cdots & 0
  \end{pmatrix}
  .
\end{equation}
The $n$-th power of matrix $X$ is an $n\times n$ block-diagonal matrix with matrices $X_{a+1}X_{a+2}\cdots X_{a+n},a\in \ZZ/n\ZZ$ on the diagonal. 
Therefore, $\mathrm{Sp}(X_1\cdots X_n)=\mathrm{Sp}(X^n)$, and we can restrict ourselves to the study of eigenvalues of the matrix $X$.

Let $s_1(\cdot)$ and $s_N(\cdot)$ denote the smallest and the largest singular values of a matrix with $N$-dependent size. Then it is sufficient to show the following.

\begin{thm}
  Let $X_1,\ldots,X_n$ be independent $N\times N$ matrices with independent centered entries (real or complex with independent real and imaginary parts) of variance $N^{-1}$ that satisfy the uniform subexponential decay condition \eqref{eq:unif_decay_cond}. 
  Define the matrix $X$ as in \eqref{eq:linearisation_matrix}.
  Let $\delta>0$.
  Then there exists $c>0$ such that almost surely for $N$ sufficiently large
  \begin{equation*}
    \inf_{|z|\geq 1+\delta} s_1(X-zI)
    \geq 
    c
    .
  \end{equation*}
\end{thm}

The next two lemmas  allow us to consider the lower bound of $s_1(X)$ only pointwise on a compact set, instead of taking the infimum over $\{|z|\geq 1+\delta\}$.

\begin{lem} Almost surely
  \begin{equation*}
    \limsup_{N\rightarrow \infty} s_N(X)
    \leq
    3
    .
  \end{equation*}
\end{lem}
\begin{proof}
  The bound can be obtained using \cite[Theorem~5.9]{BaiSilv} and the subexponential decay condition \eqref{eq:unif_decay_cond} for the entries of the matrices $X_1,\ldots,X_n$.
\end{proof}
It follows from the above lemma that it is sufficient to show that $s_1(X-z)\geq c$ holds on the set $1+\delta\leq |z|\leq 6$.
Indeed, if $|z|> 6$ then
\begin{equation*}
  s_1(X-z)
  =
  \inf_{\substack{y\in\CC^{nN}\\\|y\|=1}}\|(X-z)y\|
  \geq
  \inf_{\substack{y\in\CC^{nN}\\\|y\|=1}}(\|zy\|-\|Xy\|)
  \geq
  6-\sup_{\substack{y\in\CC^{nN}\\\|y\|=1}}\|Xy\|
  =
  6-s_N(X)
  \geq
  1
\end{equation*}
for $N$ large enough.

The next lemma shows that $s_1(X-z)$ is Lipschitz continuous with respect to $z$.
\begin{lem} For any $z,z'\in\CC$
  \begin{equation*}
  \left|s_1(X-z)-s_1(X-z')\right|
  \leq
  |z-z'|
  .
  \end{equation*}
\end{lem}
\begin{proof}
  Follows from \cite[formula~7.3.13]{HornJohn}.
\end{proof}

Therefore, using the $\varepsilon$-net argument from \cite[proof of Theorem~5.7]{OrouRenf}, we see that our initial Theorem~\ref{thm:convg_spectral_radius} can be deduced from the following theorem.

\begin{thm}\label{thm:least_eigenvalue_hermitised}
  Let $X_1,\ldots,X_n$ be independent $N\times N$ matrices with independent centered entries (real or complex with independent real and imaginary parts) of variance $N^{-1}$ that satisfy the uniform subexponential decay condition \eqref{eq:unif_decay_cond}. 
  Define the matrix $X$ as in \eqref{eq:linearisation_matrix}.
  Let $\delta>0$.
  Then there exists $c>0$ such that the following holds.
  For any $z\in\CC$ such that  $1+\delta\leq |z|\leq 6$, almost surely for $N$ sufficiently large, the interval $[0,c]$ does not contain any eigenvalue of the matrix $(X-z)^*(X-z)$.  
\end{thm}

Denote by $\lambda_1\leq\cdots\leq \lambda_{nN}$ the eigenvalues of the matrix $(X-z)^*(X-z)$ and let $\nu_{z,N}$ be its empirical spectral measure, i.e. for $A\subset \RR$
\begin{equation*}
  \nu_{z,N}(A)
  =
  \frac{1}{nN}\sum_{j=1}^{nN}1_A(\lambda_j)
  .
\end{equation*}
To study the eigenvalues of $(X-z)^*(X-z)$, we introduce its Stieltjes transform
\begin{equation*}
  m(z,w)
  =
  \frac{1}{nN}\tr ((X-z)^*(X-z)-w)^{-1}
  =
  \frac{1}{nN}\sum_{j=1}^{nN}\frac{1}{\lambda_j-w}
  , \quad 
  w\in\CC_+
  .
\end{equation*}
Before continuining, we state some results concerning the limit of $m(z,w)$ and the weak convergence of $\nu_{z,N}$.

\begin{pr}\label{thm:sosh_orou_product}
  There exist a deterministic function $m_c:\CC\times\CC\rightarrow \CC$ and a family of  deterministic measures $\{\nu_z,\, z\in\CC\}$ on $\RR_+$ such that
  \begin{itemize}
  \item[(1)] Almost surely, $m(z,w)$ converges to $m_c(z,w)$ as $N\rightarrow \infty$,
  \item[(2)] Almost surely, $\nu_{z,N}$ converges weakly to $\nu_z$ uniformly in every bounded region of $z$.
  \end{itemize}
\end{pr}

The case when all non-zero entries of the matrix $X$ are independent and \emph{identically distributed} was proven by O'Rourke and Soshnikov in \cite{OrouSosh} under the condition of finite $(2+\varepsilon)$th moment.
This proof can be easily adapted for the matrices with non-identically distributed entries satisfying the uniform subexponential decay condition, therefore we omit the details and provide only the main steps of the proof:
\begin{itemize}
\item[(i)] Firstly, following the argument of \cite[Lemma~27]{OrouSosh} we can show that in order to obtain the convergence of $\nu_{z,N}$ uniformly in bounded regions of $z$ it is enough to prove that uniformly in bounded regions of $z$ $|m(z,w)-m_c(z,w)|$ converges to zero as $N\rightarrow \infty$;
\item[(ii)] next, from the McDiarmid's concentration inequality (\cite{Mcdi}) we get that it is enough to show the convergence of the expected value $\EE[m(z,w)]$ to $m_c(z,w)$;
\item[(iii)] from Lemma~\ref{lem:gaus_diff} below we see that the limits of the expectations of the Stieltjes transforms for matrices with identically distributed and non-identically distributed entries coinside, therefore Proposition~\ref{thm:sosh_orou_product} follows now from the proofs of Theorem~15 and Lemma~27 in \cite{OrouSosh}.
\end{itemize}
We now state a lemma showing that $\EE[m(z,w)]$ can be approximated by the expectation of the Stieltjes transform of a matrix with iid entries.
Let $\hat{X}$ be $nN\times nN$ random matrix having the same block structure as the matrix $X$ but with identically distributed non-zero entries.
Suppose that $\LRI{\hat{X}}{}{12}{11}$ has the same distribution as $\LRI{{X}}{}{12}{11}$.
Denote by $\hat{G}(w,z)$ and $\LRI{m}{}{}{\hat{G}}$  corresponding resolvent matrice and Stieltjes transform. 
As this result will be later used in Section~\ref{sec:weak_concentration}, we state it in the terms of normalised partial traces of the resolvent rather than in the terms of the Stieltjes transform.
More precisely, if we have a resolvent matrices $G$ and $\hat{G}$, then for $1\leq a\leq n$ we define the normalised partial traces
\begin{equation*}
  \LRI{m}{}{a}{G}
  =
  \frac{1}{N}\sum_{i=(a-1)N+1}^{aN} G_{ii} 
  ,\quad
    \LRI{m}{}{a}{\hat{G}}
  =
  \frac{1}{N}\sum_{i=(a-1)N+1}^{aN} \hat{G}_{ii} 
  .
\end{equation*}
The result is the following lemma, proof of which can be found in the appendix.
\begin{lem}\label{lem:gaus_diff}
For any $c>0$ there exists $C>0$ such that for any $|z|\leq c$ and $w=E+\sqrt{-1}\eta$ with $0\leq E\leq c, \eta = c^{-1}$ we have
\begin{equation}
  \max_{1\leq a\leq n}|\E[\LRI{m}{}{a}{\hat{G}}]-\E[\LRI{m}{}{a}{G}]|
    \leq
    C N^{-1/2}.
\end{equation}
\end{lem}
In Proposition~\ref{thm:sosh_orou_product} we identified the limits of $m(z,w)$ and $\nu_{z,N}$.
Theorem~\ref{thm:two_condition_thm} below allows us to reduce our study of the singular values of the matrix $X-z$ to the study of $m(z,w)$ and $\nu_z$. 
This method was introduced by Bai and Silverstein and the proof can be found in \cite[Section~6.2.5]{BaiSilv} or in \cite[Section~6.5]{OrouRenf}.
For the reader's convenience we give a proof in our particular case.
\begin{thm}\label{thm:two_condition_thm}
  Fix $\delta>0$.
  Suppose the following two conditions hold
  \begin{itemize}
  \item[(i)] there exists $c>0$ such that for any $z\in\CC, |z|\geq 1+\delta$,
    \begin{equation*}
      \nu_z([0,c])=0;
    \end{equation*}
  \item[(ii)] for any $z\in\CC, 1+\delta\leq |z|\leq 6$, almost surely
    \begin{equation}\label{eq:two_condition_thm_1}
      \sup_{E\in [0,c]}|m(z,E+\sqrt{-1}N^{-1/2})-m_c(z,E+\sqrt{-1}N^{-1/2})|
      =
      \oo{\frac{1}{\sqrt{N}}}
      .
    \end{equation}
  \end{itemize}
Then, for any $z\in\CC, 1+\delta\leq|z|\leq 6$, almost surely, the interval $[0,c/2]$ does not contain any eigenvalue of $(X-z)^*(X-z)$.

\end{thm}
\begin{proof}
From condition (ii) we have that
  \begin{EA}{ll}
    |\Im (m(z,E+\sqrt{-1}N^{-1/2})-m_c(z,E+\sqrt{-1}N^{-1/2}))|
    &=
    \left|\Im\left( \int \frac{ d \nu_{z,N}(\lambda) }{E+\sqrt{-1}N^{-1/2}-\lambda}-\int \frac{d \nu_{z}(\lambda) }{E+\sqrt{-1}N^{-1/2}-\lambda}\right)\right|
    \\
    &=
    \int \frac{N^{-1/2} (d\nu_{z,N}(\lambda)-d\nu_{z}(\lambda)) }{(E-\lambda)^2+N^{-1}}
    \\
    &=
    \oo{N^{-1/2}}
    .
  \end{EA}
Therefore
\begin{equation}\label{eq:two_cond_thm_pr_2}
  \sup_{E\in [0,c]} \int \frac{d\nu_{z,N}(\lambda)-d\nu_{z}(\lambda) }{(E-\lambda)^2+N^{-1}}
  =
  \oo{1}
  .
\end{equation}
Consider separately the above integral over $[0,c]$ and $[0,c]^{\complement}$, so that
\begin{equation}\label{eq:two_cond_thm_pr_1}
    \int \frac{d\nu_{z,N}(\lambda)-d\nu_{z}(\lambda) }{(E-\lambda)^2+N^{-1}}
    =
    \int_{[0,c]} \frac{d\nu_{z,N}(\lambda) }{(E-\lambda)^2+N^{-1}}
     + \int_{[0,c]^{\complement}} \frac{d\nu_{z,N}(\lambda)-d\nu_{z}(\lambda) }{(E-\lambda)^2+N^{-1}}
     ,
\end{equation}
where we used that $\nu_z[0,c]=0$.
From the weak convergence of $\nu_{z,N}$ we have that the second integral in \eqref{eq:two_cond_thm_pr_1} converges uniformly  to zero as
\begin{equation*}
    \sup_{E\in [0,c/2]}\left|\int_{[0,c]^{\complement}} \frac{d\nu_{z,N}(\lambda)-d\nu_{z}(\lambda) }{(E-\lambda)^2+N^{-1}}\right| 
    \rightarrow 0
    ,\quad
    N\rightarrow \infty.
\end{equation*}
Using the definition of $\nu_{z,N}$ we can rewrite the first term in \eqref{eq:two_cond_thm_pr_1} as
\begin{equation*}
    \int_{[0,c]} \frac{d\nu_{z,N}(\lambda) }{(E-\lambda)^2+N^{-1}}
    =
    \frac{1}{Nn} \sum_{i:\lambda_i\in [0,c]} \frac{1}{(E-\lambda_i)^2+N^{-1}}
    .
\end{equation*}
If there exists $i\in \llbracket 1,nN\rrbracket$ such that $\lambda_i \in [0,c/2]$, then
\begin{equation*}
  \sup_{E\in [0,c/2]}
  \int_{[0,c]} \frac{d\nu_{z,N}(\lambda) }{(E-\lambda)^2+N^{-1}}
  \geq
  \frac{1}{n}
  ,
\end{equation*}
which contradicts \eqref{eq:two_cond_thm_pr_2}.
\end{proof}

The first condition in Theorem~\ref{thm:two_condition_thm} can be obtained from the known properties of $\nu_z$ (see Lemma~\ref{lem:mc_properties}, statement (1)). 

As a consequence of this discussion, the task is to show that condition (ii) in Theorem~\ref{thm:two_condition_thm} holds true, to which the rest of the article will be devoted.

\section{Notations and definitions}
\label{sec:notation}
We start by fixing the notation and giving necessary definitions.
The main argument will follow the general framework proposed by Bourgade, Yau and Yin, therefore we try to keep our notation as close as possible to the notation used in \cite{BourYauYin}.

Throughout the rest of the article $a$ and $b$ will be elements of $\ZZ/n\ZZ$. 

Let $X_a,a\in \ZZ /n \ZZ, $ be independent $N\times N$ matrices, entries of which are independent random variables (real or complex with independent real and imaginary parts) with zero mean, variance $N^{-1}$ and satisfying condition \eqref{eq:unif_decay_cond}. 
Let $X$ be defined by \eqref{eq:linearisation_matrix}.
For  $z\in\CC$ and $w=E+\sqrt{-1}\eta\in\CC_+$ introduce the matrices
\begin{equation*}
  Y_z:=X-z
  ,\quad 
  G(w):=(Y_z^*Y_z-w)^{-1}
  ,\quad
  \Gc(w):=(Y_z Y_z^*-w)^{-1}.
\end{equation*}

We shall consider $nN\times nN$ matrices as consisting of $N\times N$ blocks indexed by $(a,b)$.
We shall use left superscript to specify the submatrix.
For example, for $i,j\in\{1,\ldots,N\}$ 
\begin{equation*}
  \LRI{x}{}{ab}{ij}
  :=
  x_{(\xo{a}-1)N+i,(\xo{b}-1)N+j}
  ,\quad
  \LRI{\Gc}{}{ab}{ij}
  :=
  \Gc_{(\xo{a}-1)N+i,(\xo{b}-1)N+j}.
\end{equation*}
where $\xo{a}\in a\cap\{1,\ldots,n\}$ and $\xo{b}\in b\cap\{1,\ldots,n\}$.

Define also $i(a):=(\xo{a}-1)N+i$ for $1\leq i\leq N$.

We shall use the index $a$ instead of $aa$ for for elements of the diagonal blocks, for example, $\LRI{G}{}{a}{kl}:=G_{k(a),l(a)}$. 

The $i(a)$th rows of matrices $X$ and $Y_z$ will be denoted by $\LRI{x}{}{}{i(a)}$ and $\LRI{y}{}{}{i(a)}$ respectively.
The corresponding columns of these matrices will be denoted by $\LRI{\xb}{}{}{i(a)}$ and $\LRI{\yb}{}{}{i(a)}$.

For $\TT,\UU\subset\{1,\ldots,nN\}$, $Y_z^{(\TT,\UU)}$ will denote a $(nN-|\UU|)\times (nN-|\TT|)$ matrix, obtained from $Y_z$ by deleting the rows with indices in $\UU$ and columns with indices in $\TT$. 
Resolvent matrices, corresponding to these minors, will be denoted by $G^{(\TT,\UU)}$ and $\Gc^{(\TT,\UU)}$, i.e.,
\begin{equation*}
  G^{(\TT,\UU)}(w)
  =
  (Y_z^{(\TT,\UU)*}Y_z^{(\TT,\UU)}-w)^{-1}
  ,\quad
  \Gc^{(\TT,\UU)}(w)
  =
  (Y_z^{(\TT,\UU)}Y_z^{(\TT,\UU)*}-w)^{-1}.
\end{equation*}
Note that we shall keep the indices of the elements of matrices for minors of these matrices. 
More precisely, the entries of $Y_z^{(\TT,\UU)}$ will be indexed by $(\{1,\ldots,nN\}\setminus \UU)\times(\{1,\ldots,nN\}\setminus \TT)$, and the entries of  $G^{(\TT,\UU)}$ and $\Gc^{(\TT,\UU)}$ by $(\{1,\ldots,nN\}\setminus \TT)^2$ and $(\{1,\ldots,nN\}\setminus \UU)^2$ respectively. For $i\in\TT, j\in\UU$ and $k\in\{1,\ldots,nN\}$ we shall define
\begin{equation*}
  G^{(\TT,\UU)}_{ik}
  =G^{(\TT,\UU)}_{ki}
  =0
  ,\quad
  \Gc^{(\TT,\UU)}_{jk}
  =
  \Gc^{(\TT,\UU)}_{kj}
  =0
.
\end{equation*}
Note that all these matrices depend on $z\in\CC$ and $w\in\CC_+$.

For $a\in\ZZ/n\ZZ$ and $\TT,\UU\subset\{1,\ldots,nN\}$, define next
\begin{equation*}
  \LRI{m}{(\TT,\UU)}{a}{G}
  :=
  \frac{1}{N}\sum_{i=1}^{N}\LRI{G}{(\TT,\UU)}{a}{ii}
  ,\quad
  \LRI{m}{(\TT,\UU)}{a}{\Gc}
  :=
  \frac{1}{N}\sum_{i=1}^{N}\LRI{\Gc}{(\TT,\UU)}{a}{ii}
  .
\end{equation*}
If $\TT=\UU=\emptyset$, we shall drop the $(\emptyset,\emptyset)$ superscript and write $\LRI{m}{}{a}{G}$ or $\LRI{m}{}{a}{\Gc}$.
Note, that
\begin{equation*}
  m_G(z,w)
  =
  \frac{1}{n}\sum_{a}\LRI{m}{}{a}{G}(z,w)
  .
\end{equation*}

Now we can introduce
\begin{equation*}
  \Lambda
  :=
  \max_a \{|\LRI{m}{}{a}{G}-m_c|,|\LRI{m}{}{a}{\Gc}-m_c|\}
\end{equation*}
and
\begin{equation*}
  \Psi
  :=
  \frac{1}{\sqrt{N}}+\sqrt{\frac{\Lambda}{N\eta}}+\frac{1}{N\eta},
\end{equation*}
where $m_c$ was defined in Proposition~\ref{thm:sosh_orou_product}.

We shall use $C$ and $c$ to denote different constants, that do not depend on $N$, $w$ or $z$.

For $\zeta>0$ we say that an event $\Xi_N$ holds with $\zeta$-high probability, if 
\begin{equation*}
  \Pr{\Xi^{\complement}}
  \leq
  N^{C}e^{-\varphi^{\zeta}}
  ,
\end{equation*}
where 
\begin{equation}
  \label{eq:varphi}
  \varphi
  :=
  (\log N)^{\log \log N}
\end{equation}
and $C>0$.

For $A,B>0$ we shall write  $A\sim_c B$ or simply $A\sim B$ if there is $c>0$ such that 
\begin{equation}
  \label{eq:equivalence_def}
   c^{-1}A\leq B\leq c A.
\end{equation}


\section{Tools and Methods}
\label{sec:tools}
This section collects some basic and classical tools which will be relevant towards the main result. 
Note that in Lemmas \ref{lem:bord_guion_minor_estimate}, \ref{lem:im_of_resolvent_diag_and_product}, \ref{lem:minor_differences} and \ref{lem:derivative_estimates} we deal with objects introduced in Section~\ref{sec:notation}, while in Lemma~\ref{lem:mc_properties} we study properties of the function $m_c$, which was introduced in Proposition~\ref{thm:sosh_orou_product}.

\subsection{Linear Algebra}
\label{sec:linear_algebra}

\begin{lem}\emph{(Schur complement formula, \cite[Section~0.7.3]{HornJohn})}
  Let $A$ be an invertible matrix and let $B$ be its inverse. Divide the matrices $A$ and $B$ into blocks
  \begin{equation*}
    A=
    \begin{pmatrix}
      A_{11}&A_{12}
      \\
      A_{21}&A_{22}
    \end{pmatrix}
    ,\quad
    B=
    \begin{pmatrix}
      B_{11}&B_{12}
      \\
      B_{21}&B_{22}
    \end{pmatrix},
  \end{equation*}
so that the blocks with the same index have the same size, and the blocks on the diagonal are square submatrices.

If $A_{22}$ is invertible, then
  \begin{equation}\label{eq:schur_complement}
    \left[B_{11}\right]^{-1}
    =
    A_{11}-A_{12}\left[A_{22}\right]^{-1}A_{21}
    .
  \end{equation}
\end{lem}

\begin{lem}
  Let $A$ be a square matrix and let $w$ be a complex number. If $A^*A-\omega$ is invertible, then
    \begin{equation*}
      A(A^*A-w)^{-1}A^*
      =
      I+w(AA^* -w)^{-1}
      .
    \end{equation*}
\end{lem}

\begin{proof}
  Follows from Woodbury matrix identity (see \cite[Section~0.7.4]{HornJohn}).
\end{proof}

\begin{lem}\emph{\cite[proof of Lemma~C.3]{BordGuio}}\label{lem:bord_guion_minor_estimate}
  Let $\TT,\UU,\Kc\subset\{1,\ldots,nN\}$. Then for any $i\notin \TT$
  \begin{equation}\label{eq:bord_guion_minor_estimate}
    |\sum_{k\in
      \Kc} G^{(\TT i,\UU)}_{kk}-G^{(\TT ,\UU)}_{kk}|
    \leq
    \frac{4}{\eta}
    ,\quad
    |\sum_{k\in \Kc} G^{(\UU,\TT i)}_{kk}-G^{(\UU,\TT)}_{kk}|
    \leq
    \frac{4}{\eta}
    .
  \end{equation}
The same is true for $\Gc$.
\end{lem}

\begin{lem}\label{lem:im_of_resolvent_diag_and_product}
For $i\in\{1,\ldots,nN\}$
  \begin{equation}\label{eq:im_of_resolvent_diag_and_product}
    \sum_{k=1}^{nN}|G_{ki}|^2
    =
    \frac{\Im G_{ii}}{\eta}
    .
  \end{equation}
\end{lem}
\begin{proof}
Let $\lambda_1,\ldots,\lambda_{nN}$ be eigenvalues of the matrix $Y_z^* Y_z$. Then $G=U^* D U$, where $U$ is unitary and $D=\mathrm{diag}((\lambda_1-w)^{-1},\ldots,(\lambda_{nN}-w)^{-1})$. Note that
  \begin{equation*}
    \sum_{k=1}^{nN}|G_{ki}|^2
    =
    \left[G^* G\right]_{ii}
    =
    \left[U^*D^* D U\right]_{ii}
    .
  \end{equation*}
The lemma now follows from the relation
\begin{equation*}
  \frac{1}{\eta}\Im \frac{1}{\lambda_i-w}
  =
  \frac{1}{|\lambda_i-w|^2}
  ,\quad
  i\in\{1,\ldots,nN\}
  .
\end{equation*}
\end{proof}

\begin{lem}\emph{(\cite[Lemma~6.3]{BourYauYin})}\label{lem:minor_differences}
  Let $\TT,\UU\subset\llbracket 1,nN\rrbracket$. If $i,j\notin \TT\cup \{k\}$ then
  \begin{equation}
    \label{eq:minor_differences_1}
    G_{ij}^{(\TT,\UU)}-G_{ij}^{(\TT k,\UU)}
    =
    \frac{G_{ik}^{(\TT,\UU)} G_{kj}^{(\TT,\UU)}}{G_{kk}^{(\TT,\UU)}}
    ,\quad
    \Gc_{ij}^{(\UU,\TT)}-\Gc_{ij}^{(\UU ,\TT k)}
    =
    \frac{\Gc_{ik}^{(\UU,\TT)} \Gc_{kj}^{(\UU,\TT)}}{\Gc_{kk}^{(\UU,\TT)}}
    ,
  \end{equation}
  \begin{equation}
    \label{eq:minor_differences_2}
    G^{(\TT,\UU)}-G^{(\TT,\UU k)}
    =
    -\frac{\left(G^{(\TT,\UU k)}y_{k}^* \right) \left(y_k G^{(\TT,\UU k)}\right)}{1+y_k G^{(\TT,\UU k)} y_k^*}
    ,\quad
    \Gc^{(\TT,\UU)}-\Gc^{(\TT k,\UU)}
    =
    -\frac{\left(\Gc^{(\TT k,\UU )} \yb_{k} \right) \left(\yb_k^* \Gc^{(\TT k,\UU)} \right)}{1+\yb_k^* \Gc^{(\TT k,\UU)} \yb_k}
    .
  \end{equation}
\end{lem}

\begin{lem}\label{lem:derivative_estimates}
Let $w=E+\sqrt{-1}\eta\in\CC_+$. 
Then for any $i\in\llbracket 1,nN\rrbracket$
\begin{EA}{c}
  \left|\frac{\partial G_{ii}}{\partial E}\right|
  +
  \left|\frac{\partial G_{ii}}{\partial \eta}\right|
  +
  \left|\frac{\partial \Gc_{ii}}{\partial E}\right|
  +
  \left|\frac{\partial \Gc_{ii}}{\partial \eta}\right|
  =
  \OO{\frac{1}{\eta^2}}.
\end{EA}
\end{lem}
\begin{proof}
  With the notation used in the proof of Lemma~\ref{lem:im_of_resolvent_diag_and_product}, for any $i\in\llbracket 1,nN\rrbracket$
  \begin{equation*}
    G_{ii}
    =
    \sum_{j=1}^{nN} |u_{ij}|^2 \frac{1}{\lambda_{j}-E-i\eta}
    ,
  \end{equation*}
where $u_{ij}$ are the entries or the unitary matrix $U$.
Therefore, the bound for $G_{ii}$ follows from the fact that
\begin{equation*}
  \left|\frac{\partial}{\partial \eta}\left(\frac{1}{\lambda_{j}-E-\sqrt{-1}\eta}\right)\right|+\left|\frac{\partial}{\partial E}\left(\frac{1}{\lambda_{j}-E-\sqrt{-1}\eta}\right)\right|
  =
  \OO{\frac{1}{\eta^2}}
  .
\end{equation*}
The proof for $\Gc_{ii}$ is similar.
\end{proof}

\subsection{Properties of $m_c$}

\begin{lem}\emph{(\cite[Lemma~4.2]{BourYauYin})}\label{lem:mc_properties}
Let $m_c(z,w)$ be the solution of the equation
  \begin{equation}\label{eq:mc_equation}
    m_c^{-1}=-w(1+m_c)+|z|^2(1+m_c)^{-1}
  \end{equation}
  that satisfies $\Im m_c(z,w)>0$ if $\Im w>0$.

There exists an absolutely continuous probability measure $\nu_z$ with density $\rho_z$ such that $m_c(z,w)$ is the Stieltjes transform of measure $\nu_z$, i.e.
    \begin{equation*}
      m_c(z,w)
      =
      \int_{\RR}\frac{\rho_z(x)}{\lambda-w}dx
      .
    \end{equation*}
Let $|z|\geq 1+\delta$ for some $\delta>0$.
Then the following holds.

\begin{itemize}
\item[(1)] 
Let
  \begin{equation*}
    \afr:=\sqrt{1+8|z|^2}
    ,\quad 
    \lambda_{\pm}=\frac{(\afr\pm 3)^3}{8(\afr\pm 1)}.
  \end{equation*}
  Then the support of $\rho_z(x)$ is the interval $(\lambda_-,\lambda_+)$.
\item[(2)]
There exists $\tau_0>0$ such that if $|w-\lambda_-|\geq \tau_0$ and $E\leq \lambda_-$, then
\begin{equation}\label{eq:mc_equivalences}
  |\Re m_c|\sim 1
  ,\quad
  \Re m_c\geq 0
  ,\quad
  \Im m_c \sim \eta.
\end{equation}
\end{itemize}

\end{lem}

\subsection{Large deviation estimates}
\label{sec:lde}

\begin{lem}\emph{(\cite[Lemma~10.2]{BourYauYin})}\label{lem:lde_byy}
  Let $a_i,1\leq i\leq N$, be independent complex random variables with mean zero, variance $\sigma^2$ and having a uniform subexponential decay
  \begin{equation*}
    \Pr{|a_i|\geq x\sigma}
    \leq
    \theta^{-1}e^{-x^{\theta}}
    ,\quad
    \forall x\geq 1
    ,
  \end{equation*}
with some $\theta>0$. Let $A_i,B_{ij}\in\CC, 1\leq i,j\leq N$. 
Then there exists a constant $0<\phi<1$, depending only on $\theta$, such that for any $\xi>1$ we have
\begin{EA}{c}\EAy\label{eq:lde_thm_1}
  \Pr{\left|\sum_{i=1}^{N}a_iA_i\right|\geq (\log N)^{\xi}\sigma \left(\sum_i |A_i|^2 \right)^{1/2}}
  \leq
  e^{-(\log N)^{\phi \xi}}
  ,\\\EAy\label{eq:lde_thm_2}
  \Pr{\left|\sum_{i=1}^{N} \xo{a}_i B_{ii} a_i -\sum_{i=1}^{N} \sigma^2 B_{ii}\right|\geq (\log N)^{\xi}\sigma^2 \left(\sum_{i=1}^N |B_{ii}|^2 \right)^{1/2}}
  \leq
  e^{-(\log N)^{\phi \xi}}
  ,\\\EAy\label{eq:lde_thm_3}
  \Pr{\left|\sum_{i\neq j} \xo{a}_i B_{ij} a_j \right|\geq (\log N)^{\xi}\sigma^2 \left(\sum_{i\neq j} |B_{ij}|^2 \right)^{1/2}}
  \leq
  e^{-(\log N)^{\phi \xi}}
  ,
\end{EA}
for any sufficiently large $N>N_0$, where $N_0=N_0(\theta)$ depends on $\theta$.
\end{lem}

\subsection{McDiarmid's Concentration Inequality}

\begin{pr}\emph{(\cite{Mcdi})}
Let $U=(u_1,\ldots,u_N)$ be a family of independent random variables taking values in the set $A$.
Suppose that the real-valued function $f:A^N\rightarrow \RR$ satisfies
\begin{equation}\label{eq:mcdiarmid_condition}
  |f(u)-f(u')|\leq c_k
\end{equation}
if the vectors $u$ and $u'$ differs only in $k$th coordinate.
Then for any $t\geq 0$
\begin{equation*}
  \Pr{|f(U)-\Me{f(U)}|\geq t}
  \leq
  2 e^{-2t^2/\sum c_k^2}
  .
\end{equation*}
\end{pr}

\subsection{Abstract Decoupling Lemma}

\begin{pr} \emph{(Abstract decoupling lemma, \cite[Lemma~7.3]{PillYin})}\label{thm:abs_dec}
    Let $\mathcal{I}$ be a finite set which may depend on $N$ and let $\mathcal{I}_i\subset\mathcal{I}, 1\leq i\leq N$.  Let $\{x_{\alpha},\alpha\in\mathcal{I}\}$ be a collection of independent random variables and $S_1,\ldots,S_N$ be random variables which are functions of $\{x_{\alpha},\alpha\in\mathcal{I}\}$.
  Let $\mathbb{E}_i$ denote the expectation value operator with respect to $\{x_{\alpha},\alpha\in\mathcal{I}_i\}$.
  Define the commuting projection operators
  \begin{equation*}
    \QQ_i=1-\mathbb{E}_i,P_i=\mathbb{E}_i, P_i^2=P_i, \QQ_i^2=\QQ_i, [\QQ_i,P_j]=[P_i,P_j]=[\QQ_i,\QQ_j]=0
  \end{equation*}
  and for $\AA\subset\{1,2,\ldots,N\}$,
  \begin{equation*}
    \QQ_{\AA}:=\prod_{i\in \AA}\QQ_i, P_{\AA}:=\prod_{i\in \AA}P_i
  \end{equation*}
  We use the notation
  \begin{equation*}
    [\QQ S]=\frac{1}{N}\sum_{i=1}^{N}\QQ_iS_i
  \end{equation*}

  Let $\Xi$ be an event and $p$ an even integer, which may depend on $N$. Suppose the following assumptions hold with some constants $C_0, c_0>0$.
  \begin{itemize}
  \item[(i)]  There exist deterministic positive numbers $\mathcal{X}<1$ and $\mathcal{Y}$ such that for any set $\AA\subset\{1,2,\ldots,N\}$ with $i\in \AA$ and $|\AA|\leq p, \QQ_{\AA}S_i$ in $\Xi$ can be written as the sum of two new random variables:
  \begin{equation}\label{eq:abs_dec_cond_i}
      1(\Xi)(\QQ_{\AA}S_i)=S_{i,\AA}+1(\Xi)\QQ_{\AA} 1(\Xi^{\complement})\tilde{S}_{i,\AA}
  \end{equation}
  and
  \begin{equation*}
    |S_{i,\AA}|\leq \mathcal{Y}(C_0\mathcal{X}|\AA|)^{|\AA|}, |\tilde{S}_{i,A}|\leq \mathcal{Y}N^{C_0|\AA|}
    ;
  \end{equation*}
  \item[(ii)] 
    \begin{equation*}
      \max_{i}|S_i|\leq\mathcal{Y}N^{C_0}
      ;
    \end{equation*}
  \item[(iii)] 
    \begin{equation*}
      \Pb[\Xi^c]\leq e^{-c_0(\log N)^{3/2}p}
      .
    \end{equation*}
  \end{itemize}
  Then, under the assumptions (i), (ii), (iii) above, we have
  \begin{equation*}
    \E[\QQ S]^p\leq (Cp)^{4p}[\mathcal{X}^{2}+N^{-1}]^{p}\mathcal{Y}^p
  \end{equation*}
  for some $C>0$ and any sufficiently large $N$.
\end{pr}


\section{Concentration of $m_N$}
\label{sec:proofs2}

In this section we fix $z\in\CC, \, 1+\delta \leq |z|\leq 6$ and prove the following theorem
\begin{thm}\label{thm:concentration_main}
    There exists $c>0$ such that for any $D>0$  uniformly in $\{ w\, : \, 0\leq E\leq c, N^{-1/2} \leq\eta\leq c\}$
    \begin{equation*}
      |m(z,w)-m_c(z,w)|
      =
      \oo{\frac{1}{\sqrt{N}}+\frac{1}{N\eta}+\frac{1}{\sqrt{\eta}N^{3/4}}}
    \end{equation*}
with probability at least $1-N^{-D}$.
\end{thm}

Using the above theorem together with the Borel-Cantelli lemma we can deduce that the condition~(ii) of Theorem~\ref{thm:two_condition_thm} holds almost surely for $N$ large enough.

\subsection{System of ``self-consistent equations''}
The aim of this section is to prove Theorem~\ref{thm:sce}. 
We begin with three independent lemmas.
\begin{lem}\label{lem:sce_schur}
  For any $\TT,\UU\subset \{1,\ldots,nN\}$ and $i,j\in\{1,\ldots,nN\}\setminus\TT,i\neq j$, we have
  \begin{EA}{c}\EAy\label{eq:sce_schur_1}
    \frac{1}{G_{ii}^{(\TT,\UU)}}
    =
    -w(1+\yb_i^*\Gc^{(\TT i,\UU)}\yb_i)
    ,\quad
    G^{(\TT,\UU)}_{ij}
    =    
    -wG^{(\TT,\UU)}_{ii}G^{(\TT i,\UU)}_{jj}\left(\yb_i^*\Gc^{(\TT ij,\UU)}\yb_j\right)
    ,\\ \EAy \label{eq:sce_schur_2}
    \frac{1}{\Gc_{ii}^{(\UU,\TT)}}
    =
    -w(1+y_iG^{(\UU,\TT i)}y_i^*)   
    ,\quad
    \Gc^{(\UU,\TT)}_{ij}
    =
    -w\Gc^{(\UU,\TT)}_{ii}\Gc^{(\UU ,\TT i)}_{jj}\left(y_iG^{(\UU,\TT ij)}y_j^*\right)
    .
  \end{EA}
\end{lem}
\begin{proof}
  See \cite[Lemma~6.5]{BourYauYin}.
\end{proof}

Define a subset of $\CC$
\begin{EA}{c}
  S_0
  :=
  \{w\in\CC+\,:\,  w=E+\sqrt{-1}\eta,\, 0\leq E\leq \lambda_-(z)/2,\, 0\leq\eta\leq 1\}
  .
\end{EA}

\begin{lem}\label{lem:prop_of_mc_close_functions}
There exist $\alpha>0$ small enough and $C>0$, such that for any $h_i:\CC\times\CC_+\rightarrow \CC,i\in\{1,2\}$, for all $w\in S_0$ if
  \begin{equation*}
    \max_{i\in\{1,2\}}|h_i(z,w)-m_c(z,w)|\leq 2\alpha
  \end{equation*}
holds, then
  \begin{EA}{rll}\EAy\label{eq:prop_of_mc_close_functions_1}
   (i) & \quad &\left| 1+h_1(z,w)\right| 
   \sim_C
   |h_1(z,w)|
   \sim_C
   1
   ,\\ \EAy\label{eq:prop_of_mc_close_functions_2}
   (ii) & \quad &\left|\Im \frac{1}{w(1+h_1(z,w))}\right| 
   \leq
   C\left(\frac{\eta}{|w|^2}+\frac{|h_1(z,w)-m_c(z,w)}{|w|}\right)
   , \\ \EAy\label{eq:prop_of_mc_close_functions_3}
   (iii) & \quad &\left|w(1+h_1(z,w))-\frac{|z|^2}{1+h_2(z,w)}+\frac{1}{m_c(z,w)}\right|
   \leq
   C\max_i|h_i(z,w)-m_c(z,w)|.
  \end{EA}
\end{lem}
\begin{proof} 
  Let $\epsilon\in\{0,1\}$.
  By \eqref{eq:mc_equivalences} and the definition \eqref{eq:equivalence_def} there exists $c>0$ such that for any $w\in S_0$ 
\begin{equation*}
  c
  \leq 
  |\epsilon+m_c(z,w)|
  \leq
  c^{-1}.
\end{equation*}
Then from the triangular inequality
\begin{equation*}
  |\epsilon+m_c(z,w)|-|h_1(z,w)-m_c(z,w)|
  \leq
  |\epsilon+h_1(z,w)|
  \leq
  |\epsilon +m_c(z,w)|+|h_1(z,w)-m_c(z,w)|
\end{equation*}
\eqref{eq:prop_of_mc_close_functions_1} holds for $\alpha$ small enough.

In the rest of the proof the $(z,w)$ argument will be suppressed. 
To prove \eqref{eq:prop_of_mc_close_functions_2} rewrite the left-hand side as
\begin{equation*}
  \frac{1}{w(1+h_1)}
  =
  \frac{1}{w(1+m_c+(h_1-m_c))}
  =
  \frac{1}{w(1+m_c)}-\frac{h_1-m_c}{w(1+m_c)(1+h_1)}.
\end{equation*}
From \eqref{eq:mc_equivalences} we know that $\Re m_c \geq 0$ on $S_{0}$, so that
\begin{equation*}
  \left|\Im \frac{1}{w(1+m_c)}\right|
  =
  \frac{ \eta\Re(1+m_c)+E \Im m_c}{|w(1+m_c)|^2}.
\end{equation*}
Now \eqref{eq:prop_of_mc_close_functions_2} follows from the fact that $|1+m_c|^{-1}\sim |1+h_1|^{-1}\sim (1+\Re m_c)\sim 1$ and $\Im m_c\sim \eta$.

Rewriting the left-hand side of \eqref{eq:prop_of_mc_close_functions_3} as
\begin{equation*}
  1+h_1-\frac{|z|^2}{w(1+h_2)}+\frac{1}{wm_c}
  =
  1+m_c-\frac{|z|^2}{w(1+m_c)}+\frac{1}{wm_c} +(h_1-m_c)-\frac{|z|^2(h_2-m_c)}{w(1+m_c)(1+h_2)}
\end{equation*}
and using \eqref{eq:mc_equation} we obtain the last inequality.
\end{proof}

\begin{lem}\label{lem:lde}
Let $\zeta>0$. 
Then there exists $Q_{\zeta}>0$ such that for all sufficiently large $N$, for any $\TT,\UU\subset \{1,\ldots,nN\}$, for any $a\in\ZZ/n\ZZ$ and $\{i,j\}\subset\{1,\ldots,N\}$ such that $\{i(a),j(a)\}\subset\TT$ ($i=j$ is allowed), for $(z,w)\in S_0$  with $\zeta$-high probability
  \begin{equation}\label{eq:lde_1}
    (1-\EE_{\yb_{i(a)}\yb_{j(a)}})\left[\sum_{k,l=1}^{nN} \xo{x}_{ki(a)} \Gc^{(\TT,\UU)}_{kl} x_{lj(a)}  \right]
    =
    \OO{\varphi^{Q_{\zeta}}\left(\Psi+\frac{|\TT|+|\UU|}{N\eta}\right)}
    ,
  \end{equation}
and if $i(a)\notin \UU$, then
\begin{equation}\label{eq:lde_2}
    (1-\EE_{\yb_{i(a)}})\left[\sum_{k}^{nN} \xo{x}_{ki(a)} \Gc^{(\TT,\UU)}_{ki(a)}  \right]
    =
    \OO{\varphi^{Q_{\zeta}}\sqrt{\frac{\Im \LRI{\Gc}{(\TT,\UU)}{a}{ii}}{N\eta}}}
    ,\quad
    (1-\EE_{\yb_{i(a)}})\left[\sum_{k}^{nN} \Gc^{(\TT,\UU)}_{i(a)k} {x}_{i(a)k} \right]
    =
    \OO{\varphi^{Q_{\zeta}}\sqrt{\frac{\Im \LRI{\Gc}{(\TT,\UU)}{a}{ii}}{N\eta}}}
    .
\end{equation}
\end{lem}
The above result is valid if we take the matrix $G^{(\UU,\TT)}$ and rows $y_{i(a)}$ instead of $\Gc^{(\TT,\UU)}$ and $\yb_{i(a)}$.
\begin{proof}
Consider the case $i=j$. 
Then
\begin{equation*}
  \EE_{\yb_{i(a)}}\left[\sum_{k,l=1}^{nN} \xo{x}_{ki(a)} \Gc^{(\TT,\UU)}_{kl} x_{li(a)}  \right]
  =
  \frac{1}{N}\sum_{k=1}^{N}\LRI{\Gc}{(\TT,\UU)}{a-1}{kk}
\end{equation*}
and
\begin{EA}{lcl}
  (1-\EE_{\yb_{i(a)}})\left[\sum_{k,l=1}^{nN} \xo{x}_{ki(a)} \Gc^{(\TT,\UU)}_{kl} x_{li(a)}  \right]
  &=&
  \sum_{\substack{k,l=1\\k\neq l}}^{N} \xo{x}_{k(a-1)i(a)} \LRI{\Gc}{(\TT,\UU)}{a-1}{kl} x_{l(a-1)i(a)}
  \\
  &&
  +
  \sum_{k=1}^{N} \xo{x}_{k(a-1)i(a)} \LRI{\Gc}{(\TT,\UU)}{a-1}{kk} x_{k(a-1)i(a)} - \frac{1}{N}\sum_{k=1}^{N}\LRI{\Gc}{(\TT,\UU)}{a-1}{kk}
  .
\end{EA}
If in \eqref{eq:lde_thm_2} and \eqref{eq:lde_thm_3} we take $\xi=(\zeta\log \log N)/\phi$ and $Q_{\zeta}=\zeta/\phi$ , then we have that with $\zeta$-high probability
\begin{equation*}
  |(1-\EE_{\yb_{i(a)}})\left[\sum_{k,l=1}^{nN} \xo{x}_{ki(a)} \Gc^{(\TT,\UU)}_{kl} x_{li(a)}  \right]|
  \leq
  \frac{\sqrt{2}\varphi^{Q_{\zeta}}}{N}\left(\sum_{k,l=1}^{N}\left|\LRI{\Gc}{(\TT,\UU)}{a-1}{kl}\right|^2\right)^{1/2}
  \leq
  \sqrt{2}\sqrt{n}\varphi^{Q_{\zeta}}\left(\frac{\sum_{k,l=1}^{nN}\left|\LRI{\Gc}{(\TT,\UU)}{}{kl}\right|^2}{(nN)N}\right)^{1/2}
  .
\end{equation*}
By \eqref{eq:im_of_resolvent_diag_and_product} we have
\begin{equation*}
  \sum_{k,l=1}^{nN}\left|\LRI{\Gc}{(\TT,\UU)}{}{kl}\right|^2
  =
  \frac{1}{\eta}\sum_{k=1}^{nN} \Im \LRI{\Gc}{(\TT,\UU)}{}{kl}
  ,
\end{equation*}
and by \eqref{eq:bord_guion_minor_estimate}
\begin{equation*}
  \frac{1}{nN}\sum_{k=1}^{nN} \Im \LRI{\Gc}{(\TT,\UU)}{}{kl}-\Im m_{\Gc}
  \leq
  \frac{4(|\TT|+|\UU|)}{N\eta}.
\end{equation*}
Therefore
\begin{EA}{ll}
  |(1-\EE_{\yb_{i(a)}})\left[\sum_{k,l=1}^{nN} \xo{x}_{ki(a)} \Gc^{(\TT,\UU)}_{kl} x_{li(a)}  \right]|
  &\leq
  \sqrt{2}\sqrt{n}\varphi^{Q_{\zeta}}\left(\frac{\Im m_{\Gc}}{N\eta}+\frac{4(|\TT|+|\UU|)}{(N\eta)^2}\right)^{1/2}
  \\
  &\leq
  4\sqrt{n}\varphi^{Q_{\zeta}}\left(\sqrt{\frac{\Im m_c +|m_{\Gc}-m_c|}{N\eta}}+\frac{|\TT|+|\UU|}{N\eta}\right)
  ,
\end{EA}
and we conclude by recalling that on $S_0$ $\Im m_c \sim \eta$.

With a similar argument we can show \eqref{eq:lde_1} when $i\neq j$.

To get the estimate \eqref{eq:lde_2}, we use \eqref{eq:lde_thm_1} together with \eqref{eq:im_of_resolvent_diag_and_product} to obtain that, for example, with $\zeta$-high probability
\begin{equation*}
  |(1-\EE_{\yb_{i(a)}})\left[\sum_{k}^{nN} \xo{x}_{ki(a)} \Gc^{(\TT,\UU)}_{ki(a)}  \right]|
  \leq
  \frac{\varphi^{Q_{\zeta}}}{\sqrt{N}}\sqrt{\sum_{k=1}^{N} |\LRI{\Gc}{(\TT,\UU)}{}{k(a-1)i(a)}|^2}
  \leq
  C\varphi^{Q_{\zeta}}\sqrt{\frac{\Im \LRI{\Gc}{(\TT,\UU)}{a}{ii}}{N\eta}}.
\end{equation*}

\end{proof}
From now on we fix $\alpha$ as in Lemma~\ref{lem:prop_of_mc_close_functions} and $Q_{\zeta}$ as in Lemma~\ref{lem:lde}.

To state and prove our next result we shall need some additional notation.

For $a\in\ZZ/n\ZZ$, $i\in \llbracket 1,N\rrbracket$ and $\TT\subset\llbracket 1,nN\rrbracket$ we define
\begin{equation*}
  \LRI{Z}{(\TT)}{a}{i}
  :=
  \Mes{y_{i(a)}}{y_{i(a)} G^{(\TT,\emptyset)} y_{i(a)}^*}
  ,\quad
  \LRI{\Zc}{(\TT)}{a}{i}
  :=
  \Mes{\yb_{i(a)}}{\yb_{i(a)}^* \Gc^{(\emptyset,\TT)} \yb_{i(a)}}
  ,
\end{equation*}
and we shall suppress the right superscript if $\TT=\emptyset$.

For any $t>0$ define an $N$-dependent set
\begin{equation*}
  S_t
  :=
  \{w=E+\sqrt{-1}\eta\, | \, 0\leq E\leq \lambda_-(z)/2, \frac{\varphi^{t}}{N}\leq \eta\leq1\}
  =
  S_0\cap \left\{\eta\geq \frac{\varphi^t}{N}\right\}
  .
\end{equation*}

We are now in position to prove the following theorem.

\begin{thm}\label{thm:sce}
  For any $\zeta>0$ there exists $\tilde{Q}_{\zeta}>0$ such that the following implication is true  $\forall (z,w)\in S_{\tilde{Q}_{\zeta}}$:
\newline
if
  \begin{equation}\label{eq:alpha_condition}
    \Lambda(z,w)
    \leq
    \alpha 
  \end{equation}
holds with $\zeta$-high probability, then
  \begin{EA}{c}\EAy\label{eq:sce_1}
    \LRI{G}{(\emptyset,i(a))}{a}{ii}
    =
    \left[-w(1+\LRI{m}{}{a-1}{\Gc})\right]^{-1} + \OO{\varphi^{Q_{\zeta}}\frac{\Psi}{|w|}}
    ,\quad
    a\in\ZZ/n\ZZ,\,1\leq i\leq N
    ,\\\EAy\label{eq:sce_2}
    \LRI{\Gc}{(i(a),\emptyset)}{a}{ii}
    =
    \left[-w(1+\LRI{m}{}{a+1}{G})\right]^{-1} + \OO{\varphi^{Q_{\zeta}}\frac{\Psi}{|w|}}
    ,\quad
    a\in\ZZ/n\ZZ,\,1\leq i\leq N
    ,\\\EAy\label{eq:sce_3}
    \LRI{G}{}{a}{ii}
    =
    \left[-w(1+\LRI{m}{}{a-1}{\Gc})+\frac{|z|^2}{1+\LRI{m}{}{a+1}{G}}\right]^{-1} + \OO{\varphi^{2Q_{\zeta}}\Psi}
    ,\quad
    a\in\ZZ/n\ZZ,\,1\leq i\leq N
    ,\\\EAy\label{eq:sce_4}
    \LRI{\Gc}{}{a}{ii}
    =
    \left[-w(1+\LRI{m}{}{a+1}{G})+\frac{|z|^2}{1+\LRI{m}{}{a-1}{\Gc}}\right]^{-1} + \OO{\varphi^{2Q_{\zeta}}\Psi}
    ,\quad
    a\in\ZZ/n\ZZ,\,1\leq i\leq N
    ,\\\EAy\label{eq:sce_mg}
    \frac{1}{\LRI{m}{}{a}{G}}+w(1+\LRI{m}{}{a-1}{\Gc})-\frac{|z|^2}{1+\LRI{m}{}{a+1}{G}} 
    =
    \OO{\varphi^{2Q_{\zeta}}\Psi}
    ,\quad 
    a\in\ZZ/n\ZZ
    ,\\\EAy\label{eq:sce_mgc}
    \frac{1}{\LRI{m}{}{a}{\Gc}}+w(1+\LRI{m}{}{a+1}{G})-\frac{|z|^2}{1+\LRI{m}{}{a-1}{\Gc}} 
    =
    \OO{\varphi^{2Q_{\zeta}}\Psi}
    ,\quad
    a\in\ZZ/n\ZZ
    ,
  \end{EA}
hold with $\zeta$-high probability.
\end{thm}
\begin{proof}
We begin with equation \eqref{eq:sce_2}.
Using \eqref{eq:sce_schur_2} and taking the expectation with respect to $\yb_{i(a)}$
\begin{EA}{ll}
    \LRI{\Gc}{(i(a),\emptyset)}{a}{ii}
    &=
    \frac{1}{-w\left(1+\LRI{m}{(i(a),i(a))}{a+1}{G}+\LRI{Z}{(i(a))}{a}{i}\right)}
    \\
    &=
    \frac{1}{-w\left(1+\LRI{m}{}{a+1}{G}+(\LRI{m}{(i(a),i(a))}{a+1}{G}-\LRI{m}{}{a+1}{G})+\LRI{Z}{(i(a))}{a}{i}\right)}
    \\
    &=
    \frac{1}{-w\left(1+\LRI{m}{}{a+1}{G}\right)}+\frac{(\LRI{m}{(i(a),i(a))}{a+1}{G}-\LRI{m}{}{a+1}{G})+\LRI{Z}{(i(a))}{a}{i}}{w(1+\LRI{m}{}{a+1}{G})(1+\LRI{m}{(i(a),i(a))}{a+1}{G}+\LRI{Z}{(i(a))}{a}{i})}
    .
\end{EA}
The $i(a)$th row and column of $\LRI{G}{(i(a),i(a))}{}{}$ are equal to zero by definition. Therefore
\begin{equation*}
  y_{i(a)} \LRI{G}{(i(a),i(a))}{}{} y_{i(a)}^*
  =
  \sum_{k,l=1}^{N} x_{i(a)k(a+1)} \LRI{G}{(i(a),i(a))}{a+1}{kl}\xo{x}_{i(a)l(a+1}
\end{equation*}
and from \eqref{eq:lde_1} we have that $\left|\LRI{Z}{(i(a))}{a}{i}\right|=\OO{\varphi^{Q_{\zeta}}\Psi}$ for some $Q_{\zeta}>0$.

Suppose that $\tilde{Q}_{\zeta}>6 Q_{\zeta}$.
Then
\begin{equation*}
  \varphi^{2Q_{\zeta}}\Psi
  \leq
  \varphi^{2Q_{\zeta}}\left(\frac{1}{\sqrt{N}}+\frac{\sqrt{\alpha}}{\varphi^{6 Q_{\zeta}/2}}+\frac{1}{\varphi^{6 Q_{\zeta}}}\right)
  \leq
  \varphi^{-Q_{\zeta}}.
\end{equation*}
Recall that by \eqref{eq:bord_guion_minor_estimate}
\begin{equation*}
  |\LRI{m}{(i(a),i(a))}{a+1}{G}-\LRI{m}{}{a+1}{G}|
  \leq
  \frac{8}{N\eta}.
\end{equation*}
If $N$ is big enough, then
\begin{equation*}
  |\LRI{m}{(i(a),i(a))}{a+1}{G}-m_c|
  \leq
  2\alpha
\end{equation*}
and from \eqref{eq:prop_of_mc_close_functions_1} we get \eqref{eq:sce_2}.

We now apply \eqref{eq:sce_schur_2} to $\left[\LRI{G}{}{a}{ii}\right]^{-1}$, take expectation with respect to the column $\yb_{i(a)}$ and use \eqref{eq:sce_2}
\begin{EA}{ll}
    \frac{1}{\LRI{G}{}{a}{ii}}
    &=
    -w\left(1+\LRI{m}{}{a-1}{\Gc}+\LRI{\Zc}{}{a}{i} +|z|^2\LRI{\Gc}{(i(a),\emptyset)}{a}{ii}\right)
    \\ \EAy \label{eq:sce_5}
    &=
    -w(1+\LRI{m}{}{a-1}{\Gc})-w\LRI{\Zc}{}{a}{i} +\frac{|z|^2}{1+\LRI{m}{}{a+1}{G}} +\OO{\varphi^{Q_{\zeta}}\Psi}
    .
\end{EA}
We estimate $\LRI{\Zc}{}{a}{i}$ using Lemma~\ref{lem:lde} and \eqref{eq:sce_2} as
\begin{EA}{ll}
  \LRI{\Zc}{}{a}{i}
  &=
  (1-\EE_{\yb_{i(a)}})\left[\sum_{k,l=1}^{N} \xo{x}_{k(a-1)i(a)}\LRI{\Gc}{(i(a),\emptyset)}{a-1}{kl} x_{l(a-1)i(a)} 
    -\xo{z}\sum_{l=1}^{N} \LRI{\Gc}{(i(a),\emptyset)}{}{i(a)l(a-1)} x_{l(a-1)i(a)}
    -z\sum_{k=1}^{N} \xo{x}_{k(a-1)i(a)}\LRI{\Gc}{(i(a),\emptyset)}{}{k(a-1)i(a)} 
  \right]
  \\
  &=
  \OO{\varphi^{Q_{\zeta}}\Psi}+\OO{\varphi^{Q_{\zeta}}\sqrt{\frac{\Im \LRI{\Gc}{(i(a),\emptyset)}{a}{ii}}{N\eta}}}
  \\
  &=
  \OO{\varphi^{Q_{\zeta}}\Psi}+\OO{\varphi^{Q_{\zeta}}\sqrt{\frac{1}{N\eta}\left(\Im \frac{1}{w(1+\LRI{m}{}{a+1}{G})}+\OO{\varphi^{Q_{\zeta}}\frac{\Psi}{|w|}}\right)}}
  .
\end{EA}
Then by \eqref{eq:prop_of_mc_close_functions_2}
\begin{equation*}
  \sqrt{\frac{1}{N\eta}\left(\Im \frac{1}{w(1+\LRI{m}{}{a+1}{G})}\right)}
  =
  \OO{\sqrt{\frac{\eta}{|w|^2 N\eta}}+\sqrt{\frac{\Lambda}{|w|N\eta}}}.
\end{equation*}
We conclude that
  \begin{equation*}
    w\LRI{\Zc}{}{a}{i}
    =
    \OO{\varphi^{2Q_{\zeta}}\Psi}
  \end{equation*}
and thus 
  \begin{equation*}
    \frac{1}{\LRI{G}{}{a}{ii}}
    =
    -w(1+\LRI{m}{}{a-1}{\Gc})+\frac{|z|^2}{1+\LRI{m}{}{a+1}{G}}+\OO{\varphi^{2Q_{\zeta}}\Psi}.
  \end{equation*}
Now by \eqref{eq:prop_of_mc_close_functions_3}
\begin{equation*}
\left[-w(1+\LRI{m}{}{a-1}{\Gc})+\frac{|z|^2}{1+\LRI{m}{}{a+1}{G}}+\OO{\varphi^{2Q_{\zeta}}\Psi}\right]^{-1}
=
\left[-w(1+\LRI{m}{}{a-1}{\Gc})+\frac{|z|^2}{1+\LRI{m}{}{a+1}{G}}\right]^{-1}+\OO{\varphi^{2Q_{\zeta}}\Psi}
\end{equation*}
and the equation \eqref{eq:sce_3} is proven.

If we sum the left- and right-hand sides of \eqref{eq:sce_3} over $i\in\{1,\ldots,N\}$ and divide by $N$, we get
\begin{equation*}
  \LRI{m}{}{a}{G}
  =
  \left[-w(1+\LRI{m}{}{a-1}{\Gc}+\frac{|z|^2}{1+\LRI{m}{}{a+1}{G}})\right]^{-1}+\OO{\varphi^{2Q_{\zeta}}\Psi}
  .
\end{equation*}
Using again \eqref{eq:prop_of_mc_close_functions_3} we have
  \begin{equation*}
    \frac{1}{\LRI{m}{}{a}{G}}
    =
    -w(1+\LRI{m}{}{a-1}{\Gc}+\frac{|z|^2}{1+\LRI{m}{}{a+1}{G}})+\OO{\varphi^{2Q_{\zeta}}\Psi}
    .
  \end{equation*}

Equations \eqref{eq:sce_1}, \eqref{eq:sce_4} and \eqref{eq:sce_mgc} can be proven in the same way.
Theorem~\ref{thm:sce} is established.
\end{proof}

For the rest of the article let $\alpha$, $Q_{\zeta}$ and $\tilde{Q}_{\zeta}$ be defined as in Theorem~\ref{thm:sce}.
\subsection{Weak concentration}
\label{sec:weak_concentration}
Our goal in this section is to show Theorem~\ref{thm:continuity_argument} and its corollary.
The proof relies on Theorem~\ref{thm:sce} as well as Lemmas~\ref{pr:range_separation} and \ref{pr:large_eta_estimate} below.

Suppose that condition \eqref{eq:alpha_condition} holds i.e., for all $a\in\ZZ/n\ZZ$
\begin{equation*}
  |\LRI{m}{}{a}{G}-m_c|
  \leq
  \alpha
  ,\quad
  |\LRI{m}{}{a}{\Gc}-m_c|
  \leq
  \alpha.
\end{equation*}
Then, we can expand the summands on the left-hand sides of \eqref{eq:sce_mg} and \eqref{eq:sce_mgc} around $m_c$.
For example,
\begin{equation*}
  \frac{1}{\LRI{m}{}{a}{G}}
  =
  \frac{1}{m_c+(\LRI{m}{}{a}{G}-m_c)}
  =
  \frac{1}{m_c}-\frac{\LRI{m}{}{a}{G}-m_c}{m_c^2}+\OO{|\LRI{m}{}{a}{G}-m_c|^2}.
\end{equation*}
As a result, we obtain a system of approximate linear equations with respect to $\Delta_a:=(\LRI{m}{}{a}{G}-m_c)$ and $\Delta'_a:=(\LRI{m}{}{a}{\Gc}-m_c)$

\begin{EA}{c}
  \frac{1}{m_c} -\frac{\Delta_a}{m_c^2}+w(1+m_c) +w\Delta'_{a-1}-\frac{|z|^2}{(1+m_c)}+\frac{|z|^2}{(1+m_c)^2}\Delta_{a+1}+\OO{\Lambda^2} 
    =
    \OO{\varphi^{2Q_{\zeta}}\Psi}
    ,\\
  \frac{1}{m_c} -\frac{\Delta'_a}{m_c^2}+w(1+m_c) +w\Delta_{a+1}-\frac{|z|^2}{(1+m_c)}+\frac{|z|^2}{(1+m_c)^2}\Delta'_{a-1}+\OO{\Lambda^2} 
    =
    \OO{\varphi^{2Q_{\zeta}}\Psi}
    .
\end{EA}
Recall, that $m_c$ satisfies the self-consistent equation \eqref{eq:mc_equation}. We end up with the following linear system 

\begin{EA}{c}\EAy\label{eq:lin_sys_1}
    -\frac{\Delta_a}{m_c^2}+w\Delta'_{a-1}+\frac{|z|^2}{(1+m_c)^2}\Delta_{a+1}
    =
    \OO{\varphi^{2Q_{\zeta}}\Psi}+\OO{\Lambda^2} 
    ,\\ \EAy \label{eq:lin_sys_2}
    -\frac{\Delta'_a}{m_c^2}+w\Delta_{a+1}+\frac{|z|^2}{(1+m_c)^2}\Delta'_{a-1}
    =
    \OO{\varphi^{2Q_{\zeta}}\Psi}+\OO{\Lambda^2} 
    .
\end{EA}

We introduce the following notation:
\begin{equation*}
  \Delta
  :=
  (\Delta_1,\ldots,\Delta_n,\Delta'_1,\ldots,\Delta'_n)^T
  ,
\end{equation*}
and
\begin{equation*}
  \Gamma_1
  :=
  \begin{pmatrix}
      0 & 0 & \cdots& 0 & \gamma_2 & -\gamma_1
      \\
      -\gamma_1 & 0 & 0 & \cdots & 0 & \gamma_2
      \\
      \gamma_2 & -\gamma_1 & 0 & \cdots & 0 & 0
      \\
      0 & \gamma_2 & -\gamma_1 & 0 & \ddots & 0
      \\
        & \ddots &  & \ddots &  & 
      \\
      0 & \ldots & 0 & \gamma_2 & -\gamma_1 & 0  
    \end{pmatrix}
    ,\quad
  \Gamma
  :=
  \begin{pmatrix}
  wI_n&\Gamma_1
  \\
  \Gamma_1^T & wI_n
  \end{pmatrix}
  ,
\end{equation*}
where
\begin{equation*}
  \gamma_1
  :=
  \frac{1}{m_c^2}
  ,\quad
  \gamma_2
  :=
  \frac{|z|^2}{(1+m_c)^2}
  .
\end{equation*}

Thus we can rewrite the system \eqref{eq:lin_sys_1}-\eqref{eq:lin_sys_2} as
\begin{equation}\label{eq:lin_sys_3}
  \Gamma \Delta
  =
  \OO{\varphi^{2Q_{\zeta}}\Psi}+\OO{\Lambda^2} 
  .
\end{equation}

\begin{lem}\label{pr:gamma_norm_bound}
  There exists $\tau>0$ such that $\forall \zeta>0$
  \begin{equation*}
    \sup_{1+\delta\leq |z|\leq 6}\sup_{w\in \tilde{S}_{\tilde{Q}_{\zeta}} } \|\Gamma^{-1}\|
    \leq
    \tau^{-1},
  \end{equation*}
where $\tilde{S}_{\tilde{Q}_{\zeta}}:=S_{\tilde{Q}_{\zeta}}\cap\{|w|\leq \tau\}$.
\end{lem}
\begin{proof} 
First of all note that from \eqref{eq:mc_equivalences} we have that
  \begin{equation*}
    \gamma_1
    \sim
    \gamma_2
    \sim
    1.
  \end{equation*}
Thus
  \begin{equation*}
    \|\Gamma^{-1}\|
    =
    \OO{\frac{1}{|\det \Gamma|}}
  \end{equation*}
and we need to estimate $|\det \Gamma|$ from below.
From the formula for block matrices
  \begin{equation*}
    \det \Gamma
    =
    \det (w^2 I_n-\Gamma_1^T\Gamma_1).
  \end{equation*}
The matrix $w^2 I_n-\Gamma_1^T\Gamma_1$ is a circulant matrix, so we have simple formulas for its eigenvalues.
More precisely (see \cite[formula~2.2.9]{HornJohn}), if we have a circulant matrix with coefficients $c_0,c_1,\ldots,c_{n-1}$, then the eigenvalues are
  \begin{equation*}
    l_j(\mathrm{Circulant}(c_0,c_1,\ldots,c_{n-1}))
    =
    \sum_{k=0}^{n-1} c_k e^{2\pi\sqrt{-1}jk/n}
    ,\quad
    j=0,\ldots,n-1
    .
  \end{equation*}
In our case
  \begin{equation*}
    w^2 I_n-\Gamma_1^T\Gamma_1
    =
    \mathrm{Circulant}(w^2-\gamma_1^2-\gamma_2^2,\gamma_1\gamma_2,0,0,\ldots,0,\gamma_1\gamma_2)
    ,
  \end{equation*}
therefore,
  \begin{EA}{rcl}
    l_j(w^2 I_n-\Gamma_1^T\Gamma_1)
    &=&
    w^2-\gamma_1^2-\gamma_2^2 + \gamma_1\gamma_2 e^{2\pi\sqrt{-1}j/n}+\gamma_1\gamma_2 e^{2\pi\sqrt{-1}j(n-1)/n}
    \\
    &=&
    w^2-\gamma_1^2-\gamma_2^2 + 2\gamma_1\gamma_2 \Re e^{2\pi\sqrt{-1}j/n}
    .
  \end{EA}
From \eqref{eq:mc_equation}
 \begin{equation*}
    \gamma_1
    =
    \frac{1+m_c}{m_c}\gamma_2+\OO{|w|}
    ,
  \end{equation*}
so that
  \begin{EA}{rcl}
    l_j
    &=&
    -\gamma_2^2\left(\frac{1+m_c}{m_c}\right)^2-\gamma_2^2 +2\gamma_2^2\frac{1+m_c}{m_c}\Re e^{2\pi\sqrt{-1}j/n}+\OO{|w|}
    \\
    &=&
    -\gamma_2^2\left(\frac{1+m_c}{m_c}\right)\left( \frac{1+m_c}{m_c}+\frac{m_c}{1+m_c}-2\Re e^{2\pi\sqrt{-1}j/n}\right) +\OO{|w|}
    \\
    &=&
    -\gamma_2^2\left(\frac{1+m_c}{m_c}\right)\left( \frac{1}{m_c}-\frac{1}{1+m_c}+2-2\Re e^{2\pi\sqrt{-1}j/n}\right) +\OO{|w|}
    .
  \end{EA}
Using again \eqref{eq:mc_equation} we have
  \begin{equation*}
    l_j
    =
    -\gamma_2^2\left(\frac{1+m_c}{m_c}\right)\left( \frac{|z|^2-1}{1+m_c}+2-2\Re e^{2\pi\sqrt{-1}j/n}\right) +\OO{|w|}.
  \end{equation*}
Since by \eqref{eq:mc_equivalences} $|m_c|\sim|1+m_c|\sim|\gamma_2|$, it is enough to show that $\Re \frac{|z|^2-1}{1+m_c}+2-2\Re e^{2\pi\sqrt{-1}j/n}>0$.
As the real part of $m_c$ is positive on $S_0$ we can see that
  \begin{equation*}
    \Re \frac{|z|^2-1}{1+m_c}+2-2\Re e^{2\pi\sqrt{-1}j/n}
    =
    (|z|^2-1)\frac{1+\Re m_c}{|1+m_c|^2}+2-2\Re e^{2\pi\sqrt{-1}j/n}
    \geq
    \tilde{C_0}
    .
  \end{equation*}
Thus, if we take $|w|$ small enough, we have 
  \begin{equation*}
    |l_j|
    \geq
    C_0
    ,\quad
    1\leq j\leq n.
  \end{equation*}
The proof is now complete.
\end{proof}

Thanks to Lemma~\ref{pr:gamma_norm_bound} we can rewrite \eqref{eq:lin_sys_3} for $w\in\tilde{S}_{\tilde{Q}_{\zeta}}$ as
\begin{equation}\label{eq:lin_sys_4}
  \Delta
  =
  \OO{\varphi^{2Q_{\zeta}}\Psi}+\OO{\Lambda^2}.
\end{equation}

Suppose that $\Lambda=\OO{\varphi^{Q_{\zeta}}\sqrt{\Psi}}$.
Then from the above equation for all $a\in \ZZ/n\ZZ$
\begin{equation*}
  \Delta_a
  =
  \OO{\varphi^{2Q_{\zeta}}\Psi}
  ,\quad
  \Delta'_a
  =
  \OO{\varphi^{2Q_{\zeta}}\Psi}.
\end{equation*}
Thus, we have the following result.
\begin{lem}\label{pr:range_separation}
  Let $\zeta>0$. Suppose that condition
  \begin{equation*}
    \Lambda
    \leq
    \alpha
  \end{equation*}
  holds on $\tilde{S}_{\tilde{Q}_{\zeta}}$ with $\zeta$-high probability.
  Then the implication
  \begin{equation*}
    \Lambda 
    =
    \OO{\varphi^{Q_{\zeta}}\sqrt{\Psi}}
    \quad \Rightarrow \quad
    \Lambda
    =
    \OO{\varphi^{2Q_{\zeta}}\Psi}
  \end{equation*}
  holds with $\zeta$-high probability $\forall w\in\tilde{S}_{\tilde{Q}_{\zeta}}$.
\end{lem}

We now consider the case of $w\in\tilde{S}_{\tilde{Q}_{\zeta}}$ with $\eta = \OO{1}$ and we show that $\Lambda\prec N^{-1/2}$ if $\eta$ is bounded away from zero.

\begin{lem}\label{pr:large_eta_estimate}
  For any $\zeta>0$  with $\zeta$-high probability
  \begin{equation*}
    \sup_{w\in\tilde{S}_{\tilde{Q}_{\zeta}}\cap\{\eta=\tau/2\}}\Lambda
    \leq
    \varphi^{Q_{\zeta}}\frac{1}{\sqrt{N}}.
  \end{equation*}
\end{lem}
\begin{proof}
First of all recall that by \eqref{eq:bord_guion_minor_estimate}
  \begin{equation*}
    |\LRI{m}{}{a}{G}-\LRI{m}{(i,\emptyset)}{a}{G}|
    \leq
    \frac{4}{N\eta}
    ,\quad
    |\LRI{m}{}{a}{\Gc}-\LRI{m}{(i,\emptyset)}{a}{\Gc}|
    \leq
    \frac{4}{N\eta}.
  \end{equation*}
Therefore, $\LRI{m}{}{a}{G}$ and $\LRI{m}{}{a}{\Gc}$ as functions of the rows $x_{k},1\leq k\leq nN$ satisfy the condition \eqref{eq:mcdiarmid_condition} for any $w\in \tilde{S}_{\tilde{Q}_{\zeta}}\cap\{\eta=\tau/2\}$ and we can apply the McDiarmid's concentration inequality, so that
  \begin{equation*}
    \Pr{|\LRI{m}{}{a}{G}-\Me{\LRI{m}{}{a}{G}}|\geq t}
    \leq
    C e^{-c t^2 N}
  \end{equation*}
and similarly for $\LRI{m}{}{a}{\Gc}$.
If we take $t=c^{-1/2}\varphi^{\zeta/2}N^{-1/2}$ in the above inequality we get that for any $w\in \tilde{S}_{\tilde{Q}_{\zeta}}\cap\{\eta=\tau/2\}$
  \begin{equation*}
    |\LRI{m}{}{a}{G}-\Me{\LRI{m}{}{a}{G}}|
    =
    \OO{\frac{\varphi^{\zeta/2}}{\sqrt{N}}}
    ,\quad
    |\LRI{m}{}{a}{\Gc}-\Me{\LRI{m}{}{a}{\Gc}}|
    =
    \OO{\frac{\varphi^{\zeta/2}}{\sqrt{N}}}
  \end{equation*}
with $\zeta$-high probability.
In Lemma~\ref{lem:gaus_diff} we proved that
\begin{equation*}
    |\EE[\LRI{m}{}{a}{G}]-\Me{\LRI{m}{}{a}{\hat{G}}}|
    =
    \OO{\frac{\varphi^{\zeta/2}}{\sqrt{N}}}
    ,\quad
    |\EE[\LRI{m}{}{a}{\Gc}]-\Me{\LRI{m}{}{a}{\hat{\Gc}}}|
    =
    \OO{\frac{\varphi^{\zeta/2}}{\sqrt{N}}}
    ,
  \end{equation*}
where $\hat{G}$ and $\hat{\Gc}$ are the resolvent matrices corresponding to the matrix $\hat{X}$ having iid non-zero entries.
From \cite[Lemma~14]{OrouSosh} we know that
  \begin{equation*}
  \Me{\LRI{m}{}{a}{\hat{G}}}
  =
  \Me{m_{\hat{G}}}
  ,\quad
  \Me{\LRI{m}{}{a}{\hat{\Gc}}}
  =
  \Me{m_{\hat{\Gc}}}
  .
  \end{equation*}
Therefore,
\begin{equation}\label{eq:large_eta_1}
  |\LRI{m}{}{a}{G}-m_G|
  =
  \OO{\frac{\varphi^{\zeta/2}}{\sqrt{N}}}
  =
  \OO{\frac{\varphi^{Q_{\zeta}}}{\sqrt{N}}}
  ,\quad
  |\LRI{m}{}{a}{\Gc}-m_G|
  =
  \OO{\frac{\varphi^{\zeta/2}}{\sqrt{N}}}
  =
  \OO{\frac{\varphi^{Q_{\zeta}}}{\sqrt{N}}}
  ,
\end{equation}
where we used that $m_G=m_{\Gc}$ and that by definition of $Q_{\zeta}$ (see the proof of Lemma~\ref{lem:lde}) $Q_{\zeta}>\zeta$.
With the same argument as in Theorem~\ref{thm:sce} we can thus show that with $\zeta$-high probability
  \begin{equation*}
    \frac{1}{G_{ii}}
    =
    -w(1+m_G)+\frac{|z|^2}{1+m_G+\OO{\varphi^{Q_{\zeta}}N^{-1/2}}}+\OO{\varphi^{2Q_{\zeta}}N^{-1/2}}
    .
  \end{equation*}
Indeed,
\begin{equation*}
  \frac{1}{G_{ii}}
  =
   -w(1+\LRI{m}{(i(a),\emptyset)}{a-1}{\Gc})+\OO{\varphi^{2Q_{\zeta}}N^{-1/2}}+\frac{|z|^2}{1+\LRI{m}{(i(a),i(a))}{a+1}{G}+\OO{\varphi^{Q_{\zeta}}N^{-1/2}}}
   .
\end{equation*}
But from the relations \eqref{eq:bord_guion_minor_estimate} and \eqref{eq:large_eta_1} we have
\begin{EA}{ll}
  \LRI{m}{(i(a),\emptyset)}{a-1}{\Gc}
  &= m_G +(m_{\Gc}-m_G)+(\LRI{m}{}{a-1}{\Gc}-m_{\Gc})+(\LRI{m}{(i(a),\emptyset)}{a-1}{\Gc}-\LRI{m}{}{a-1}{\Gc})
  \\
  &=m_G + 0 + \OO{\frac{\varphi^{Q_{\zeta}}}{\sqrt{N}}} + \OO{\frac{1}{N}} 
  =
  m_G+\OO{\frac{\varphi^{Q_{\zeta}}}{\sqrt{N}}}
  ,
\end{EA}
and similarly $\LRI{m}{(i(a),i(a))}{a+1}{G}=m_G+\OO{\varphi^{Q_{\zeta}}N^{-1/2}}$.

Consider now the real part of $m_G$. 
For $E\leq \lambda_{-}(z)/2$ we decompose $\Re m_G$ in the following way
  \begin{equation*}
    \Re m_G
    =
    \frac{1}{nN}\sum_{j=1}^{nN}\frac{\lambda_j-E}{(\lambda_j-E)^2+\eta^2}
    =
    \frac{1}{nN}\sum_{j:\lambda_j<E}\frac{\lambda_j-E}{(\lambda_j-E)^2+\eta^2}
    +\frac{1}{nN}\sum_{j:\lambda_j\geq E}\frac{\lambda_j-E}{(\lambda_j-E)^2+\eta^2}
    .
  \end{equation*}
From the Proposition~\ref{thm:sosh_orou_product} (statement~(2)) we see that as $N$ goes to infinity the first term goes to zero
\begin{EA}{ll}
  |\frac{1}{nN}\sum_{j:\lambda_j<E}\frac{\lambda_j-E}{(\lambda_j-E)^2+\eta^2}|
  &\leq
  \frac{1}{nN}\sum_{j:\lambda_j<E}\frac{\lambda_{-}/2}{(\lambda_j-E)^2+\eta^2}
  \\
  &\leq
  \frac{1}{nN}\sum_{j:\lambda_j<E}\frac{\lambda_-}{2\eta^2}
  \\
  &\leq
  \frac{\lambda_-}{2\eta^2}\nu_{z,N}([0,\lambda_-/2])\rightarrow 0
  ,
\end{EA}
while the second is positive and bounded from below
\begin{EA}{ll}
  \frac{1}{nN}\sum_{j:\lambda_j>E}\frac{\lambda_j-E}{(\lambda_j-E)^2+\eta^2}
  &\geq
  \frac{1}{nN}\sum_{j:\lambda_-\leq \lambda_j\leq \lambda_+}\frac{\lambda_j-E}{(\lambda_j-E)^2+\eta^2}
  \\
  &\geq 
  \frac{1}{nN}\sum_{j:\lambda_-\leq \lambda_j\leq \lambda_+}\frac{\lambda_-/2}{(\lambda_+)^2+\eta^2}
  \\
  &=
  \frac{\lambda_-/2}{(\lambda_+)^2+\eta^2} \nu_{z,N}([\lambda_-,\lambda_+])\rightarrow \frac{\lambda_-/2}{(\lambda_+)^2+\eta^2}>0
  .
\end{EA}
Therefore, for $N$ large enough, $\Re m_G>0$, so that $|1+m_G| \geq 1$ and thus
\begin{equation*}
  \frac{|z|^2}{1+m_G+\OO{\varphi^{Q_{\zeta}}N^{-1/2}}}
  =
  \frac{|z|^2}{1+m_G}+\OO{\varphi^{Q_{\zeta}}N^{-1/2}}.
\end{equation*}
As in the proof of Theorem~\ref{thm:sce} in order to get the approximate equation for $G_{ii}$ we show that
\begin{equation*}
  |-w(1+m_G)+\frac{|z|^2}{1+m_G}|
\end{equation*}
is bounded from below.
Indeed,
\begin{equation*}
  \Im \left(w(1+m_G)-\frac{|z|^2}{1+m_G}\right)
  =
  E \Im m_G +\eta \Re(1+m_G)+\frac{|z|^2\Im m_G}{|1+m_G|^2}  
  \geq
  \eta.
\end{equation*}
Thus we easily deduce that
\begin{equation*}
  \frac{1}{m_G}
  =
  -w(1+m_G)+\frac{|z|^2}{1+m_G}+\OO{\varphi^{2Q_{\zeta}}N^{-1/2}}.
\end{equation*}
Now we can conclude as in \cite[Lemma~6.12]{BourYauYin} that 
\begin{equation*}
  \sup_{w\in  \tilde{S}_{\tilde{Q}_{\zeta}}\cap\{\eta=\tau/2\}}|m_G(z,w)-m_c(z,w)|
  =
  \OO{\varphi^{2Q_{\zeta}}N^{-1/2}}
\end{equation*}
with $\zeta$-high probability. The result follows using \eqref{eq:large_eta_1}.
\end{proof}

\begin{thm}\label{thm:continuity_argument}
  For any $\zeta>0$  
  \begin{equation}\label{eq:weak_concentration}
    \sup_{w\in\tilde{S}_{\tilde{Q}_{\zeta}}}\Lambda(z,w)
    =
    \OO{\varphi^{2Q_{\zeta}}\Psi}
  \end{equation}
holds with $\zeta$-high probability.
\end{thm}
\begin{proof}
Following the approach used by Bourgade, Yau and Yin in \cite{BourYauYin}, we prove the theorem in two steps.
Firstly we show that  with $\zeta$-high probability the bound 
\begin{equation}\label{eq:weak_concentration1}
  \Lambda(z,w)
    =
    \OO{\varphi^{2Q_{\zeta}}\Psi}
\end{equation}
holds for $N^{-K}$-net in $\tilde{S}_{\tilde{Q}_{\zeta}}$ for $K>0$ big enough.
Next, we use the continuity properties of $\Lambda$ to extend the result to the whole set  $\tilde{S}_{\tilde{Q}_{\zeta}}$.

First of all we note that if $\Lambda \leq C \varphi^{2Q_{\zeta}}\Psi$ with $C\geq 1$, then 
\begin{equation}\label{eq:psi_weak_estimate}
  \Lambda
  \leq 
  C^2 \varphi^{4Q_{\zeta}}\left(\frac{1}{\sqrt{N}}+\frac{1}{N\eta}\right)
  .
\end{equation}
By definition 
\begin{equation*}
  \Psi
  =
  \frac{1}{\sqrt{N}}+\frac{1}{N\eta}+\sqrt{\frac{\Lambda}{N\eta}}
  .
\end{equation*}
Therefore, 
\begin{equation*}
  \Lambda \leq C \varphi^{2Q_{\zeta}}\Psi
  \quad
  \Leftrightarrow
  \quad
  \sqrt{\Lambda}\left(\sqrt{\Lambda}-C\frac{\varphi^{2Q_{\zeta}}}{\sqrt{N\eta}}\right)
  \leq
  C\varphi^{2Q_{\zeta}}\left(\frac{1}{\sqrt{N}}+\frac{1}{N\eta}\right)
  .
\end{equation*}
After some elementary calculations we get \eqref{eq:psi_weak_estimate}.

Suppose now that
\begin{equation*}
  \Lambda(E+\sqrt{-1}\eta)=\OO{\varphi^{2Q_{\zeta}}\Psi}
\end{equation*}
for $E+\sqrt{-1}\eta \in \tilde{S}_{\tilde{Q}_{\zeta}}$ such that $E+\sqrt{-1}(\eta-N^{-K})\in\tilde{S}_{\tilde{Q}_{\zeta}}$.
Then
\begin{EA}{rcl}
  \Lambda(E+\sqrt{-1}(\eta-N^{-K}))
  &\leq&
  \Lambda(E+\sqrt{-1}\eta) +|\Lambda(E+\sqrt{-1}(\eta-N^{-K}))-\Lambda(E+\sqrt{-1}(\eta))|
  \\
  &\leq&
  \OO{\varphi^{2Q_{\zeta}}\left(\frac{1}{\sqrt{N}}+\frac{1}{N\eta}\right)} + N^{-K}\sup_{w\in\tilde{S}_{\tilde{Q}_{\zeta}}}\max_{a\in \ZZ/n\ZZ}\left|\frac{\partial (|\LRI{m}{}{a}{G}-m_c|+|\LRI{m}{}{a}{\Gc}-m_c|)}{\partial \eta}\right|
  .
\end{EA}
According to Lemma~\ref{lem:derivative_estimates} if we take $K>0$ big enough then $\exists N_0\in\NN$ such that for all $N\geq N_0$
\begin{equation}\label{eq:weak_concentration2}
  \Lambda(E+\sqrt{-1}(\eta-N^{-K}))
  \leq
  \alpha
\end{equation}
and moreover
\begin{equation}\label{eq:weak_concentration3}
  \Lambda(E+\sqrt{-1}(\eta-N^{-K}))
  \leq
  \varphi^{Q_{\zeta}}\sqrt{\Psi(E+\sqrt{-1}(\eta-N^{-K}))}.
\end{equation}
Note that for the last inequality we used that $\tilde{Q}_{\zeta}>6 Q_{\zeta}$.
By \eqref{eq:weak_concentration2} we see that Lemma~\ref{pr:range_separation} can be applied to $\Lambda(E+\sqrt{-1}(\eta-N^{-K}))$, and by \eqref{eq:weak_concentration3} we see that
\begin{equation*}
  \Lambda(E+\sqrt{-1}(\eta-N^{-K}))
  =
  \OO{\varphi^{2Q_{\zeta}}\Psi(E+\sqrt{-1}(\eta-N^{-K}))}.
\end{equation*}

We showed that $\exists K>0$ such that if \eqref{eq:weak_concentration1} holds for $E+\sqrt{-1}\eta\in\tilde{S}_{\tilde{Q}_{\zeta}}$ with $\zeta$-high probability, then with $\zeta$-high probability \eqref{eq:weak_concentration1} holds for $E+\sqrt{-1}(\eta-N^{-K})$ with the same constant in $\OO{\, }$, as long as $E+\sqrt{-1}(\eta-N^{-K})\in\tilde{S}_{\tilde{Q}_{\zeta}}$.

Let $K>0$ and let $\Theta_N(K):=\{k N^{-K}+\sqrt{-1} l N^{-K} \,|\,k,l\in\ZZ\}\subset\CC$. 
From Lemma~\ref{pr:large_eta_estimate} we know that \eqref{eq:weak_concentration1} holds for any $w\in\tilde{S}_{\tilde{Q}_{\zeta}}\cap\{\eta = \tau/2\}$.

Starting from $w\in\tilde{S}_{\tilde{Q}_{\zeta}}\cap \Theta(K)$ which are close to $\{\eta=\tau/2\}$, we can step by step decrease the imaginary part of $w$ and  show that the bound \eqref{eq:weak_concentration1} holds for all $w\in\tilde{S}_{\tilde{Q}_{\zeta}}$ with $\zeta$-high probability.

We can finish the proof by using Lemma~\ref{lem:derivative_estimates} and continuity properties of $m_c$ to extend \eqref{eq:weak_concentration1} to the whole set $\tilde{S}_{\tilde{Q}_{\zeta}}$.
\end{proof}

\begin{cor}\label{cor:weak_concentration_entries}
Let $\TT,\UU\subset \{1,\ldots,nN\}$ such that $|\TT|+|\UU|\leq p$ for some $p>0$.  
For any $\zeta>0$ for any $w\in\tilde{S}_{\tilde{Q}_{\zeta}}$ with $\zeta$-high probability the following holds

  \begin{EA}{ll} \EAy \label{eq:weak_concentration_entries_1}
    \LRI{G}{(\TT,\UU)}{a}{ii}
    =
    \frac{1}{-w(1+m_c)}+\OO{\varphi^{2Q_{\zeta}}\frac{\Psi}{|w|}}
    ,
    &\quad
    \mbox{ if }
    \quad
    i(a)\in \UU
    ,
    \\ \EAy \label{eq:weak_concentration_entries_2}
    \LRI{G}{(\TT,\UU)}{a}{ii}
    =
    m_c+\OO{\varphi^{2Q_{\zeta}}\Psi}
    ,
    &\quad
    \mbox{ if }
    \quad
    i(a)\notin \UU
    ,
    \\ \EAy \label{eq:weak_concentration_entries_3}
    G_{kl}^{(\TT,\UU)}
    =
    \OO{\varphi^{2Q_{\zeta}}\frac{\Psi}{|w|}}
    ,
    &\quad
    \mbox{ if }
    \quad
    \{k,l\}\cap \UU\neq \emptyset
    ,
    k\neq l
    ,
    \\ \EAy \label{eq:weak_concentration_entries_4}
    G_{kl}^{(\TT,\UU)}
    =
    \OO{\varphi^{2Q_{\zeta}}\Psi}
    ,
    &\quad
    \mbox{ if }
    \quad
    \{k,l\}\cap \UU= \emptyset
    ,
    k\neq l
    .
  \end{EA}

\end{cor}
\begin{proof}
From \eqref{eq:sce_schur_1} 
\begin{equation*}
  \LRI{G}{(\TT,\UU)}{a}{ii}
  =
  \frac{1}{-w\left(1+\LRI{m}{(\TT i(a),\UU)}{a-1}{\Gc}+(1-\EE_{\yb_{i(a)}})\left[\yb_{i(a)}^* \Gc^{(\TT i(a),\UU)}\yb_{i(a)}\right]\right)}
  .
\end{equation*}
By \eqref{eq:bord_guion_minor_estimate} and Theorem~\ref{thm:continuity_argument}
\begin{equation*}
  |\LRI{m}{(\TT i(a),\UU)}{a-1}{\Gc}-m_c|
  =
  \OO{\varphi^{2Q_{\zeta}}\Psi}+\OO{\frac{|\TT|+|\UU|}{N\eta}}.
\end{equation*}
Also note that since $i(a)\in \UU$
\begin{equation*}
  \yb_{i(a)}^* \Gc^{(\TT i(a),\UU)}\yb_{i(a)}
  =
  \sum_{k,l=1}^{nN} \xo{x}_{ki(a)}\Gc^{(\TT i(a),\UU)}_{kl} x_{li(a)}
\end{equation*}
thus from \eqref{eq:lde_1} and the condition $|\TT|+|\UU|\leq p$ we get that with $\zeta$-high probability
\begin{equation*}
  \LRI{G}{(\TT,\UU)}{a}{ii}
  =
  \frac{1}{-w\left(1+m_c+\OO{\varphi^{2Q_{\zeta}}\Psi}\right)}.
\end{equation*}
From \eqref{eq:mc_equivalences} we obtain \eqref{eq:weak_concentration_entries_1}.

We now prove \eqref{eq:weak_concentration_entries_2}.
From \eqref{eq:sce_schur_1} we have
\begin{equation*}
  \LRI{G}{(\TT,\UU)}{a}{ii}
  =
  -w^{-1}\left(1+\LRI{m}{(\TT i(a),\UU)}{a-1}{\Gc}+(1-\EE_{\yb_{i(a)}})\left[\yb_{i(a)}^* \Gc^{(\TT i(a),\UU)}\yb_{i(a)}\right] +|z|^2 \LRI{\Gc}{(\TT i(a),\UU)}{a}{ii} \right)^{-1}.
\end{equation*}
Using again \eqref{eq:bord_guion_minor_estimate} and Theorem~\ref{thm:continuity_argument} we have that
\begin{equation*}
  \LRI{m}{(\TT i(a),\UU)}{a-1}{\Gc}
  =
  m_c +\OO{\varphi^{2Q_{\zeta}}\Psi}+\OO{\frac{|\TT|+|\UU|}{N\eta}}.
\end{equation*}
To estimate $(1-\EE_{\yb_{i(a)}})\left[\yb_{i(a)}^* \Gc^{(\TT i(a),\UU)}\yb_{i(a)}\right]$ we apply Lemma~\ref{lem:lde} together with \eqref{eq:weak_concentration_entries_1} and \eqref{eq:prop_of_mc_close_functions_2} so that
\begin{equation*}
  (1-\EE_{\yb_{i(a)}})\left[\yb_{i(a)}^* \Gc^{(\TT i(a),\UU)}\yb_{i(a)}\right]
  =
  \OO{\varphi^{2Q_{\zeta}}\frac{\Psi}{|w|}}.
\end{equation*}
Hence we have 
\begin{equation*}
  \LRI{G}{(\TT,\UU)}{a}{ii}
  =
  -\left(w(1+m_c) +\frac{|z|^2}{1+m_c}+\OO{\varphi^{2Q_{\zeta}}\Psi} \right)^{-1}
\end{equation*}
and from \eqref{eq:mc_equivalences} we get \eqref{eq:weak_concentration_entries_2}.

To show the last two estimates we note that from \eqref{eq:sce_schur_1}
\begin{equation*}
  G^{(\TT,\UU)}_{kl}
  =    
  -wG^{(\TT,\UU)}_{kk}G^{(\TT k,\UU)}_{ll}\left(\yb_k^*\Gc^{(\TT kl,\UU)}\yb_l\right).
\end{equation*}

Consider the case $\{k,l\}\subset\UU$.
Then from \eqref{eq:weak_concentration_entries_2}
\begin{equation*}
  G^{(\TT,\UU)}_{kl}
  =    
  \OO{w^{-1}\left(\yb_k^*\Gc^{(\TT kl,\UU)}\yb_l\right)}.
\end{equation*}
Note that in this case 
\begin{equation*}
  \yb_k^*\Gc^{(\TT kl,\UU)}\yb_l
  =
  \sum_{i,j=1}^{nN} \xo{x}_{ik}\Gc^{(\TT kl,\UU)}_{ij} x_{jl}
\end{equation*}
and therefore, using \eqref{eq:lde_1}, we conclude that with $\zeta$-high probability
\begin{equation*}
    G^{(\TT,\UU)}_{kl}
  =    
  \OO{\varphi^{2Q_{\zeta}}\frac{\Psi}{|w|}}.
\end{equation*}

With the same argument we can show \eqref{eq:weak_concentration_entries_3} for the other cases.

Similarly, if $\{k,l\}\cap\UU=\emptyset$, then 
\begin{equation*}
  G^{(\TT,\UU)}_{kl}
  =    
  \OO{w\left((1-\EE_{\yb_k \yb_l})\left[\yb_k^*\Gc^{(\TT kl,\UU)}\yb_l\right]+\Gc^{(\TT kl,\UU)}_{kl}\right)}.
\end{equation*}
From Lemma~\ref{lem:lde} with $\zeta$-high probability
\begin{equation*}
  (1-\EE_{\yb_k \yb_l})\left[\yb_k^*\Gc^{(\TT kl,\UU)}\yb_l\right]
  =
  \OO{\varphi^{2Q_{\zeta}}\frac{\Psi}{|w|}},
\end{equation*}
and thus from \eqref{eq:weak_concentration_entries_3} we can deduce \eqref{eq:weak_concentration_entries_4}.

\end{proof}

\subsection{Strong concentration}
\label{sec:strong_concentration}
In this section we improve the bound \eqref{eq:weak_concentration} on a subset of $\tilde{S}_{\tilde{Q}_{\zeta}}$ and prove Theorem~\ref{thm:concentration_main}.

Firstly, similarly to the argument used by Bourgade, Yau and Yin to obtain the second order self-consistent equation (see \cite[Lemma~7.2]{BourYauYin}), we use Theorem~\ref{thm:sce} and Corollary~\ref{cor:weak_concentration_entries} to linearise equations \eqref{eq:sce_5} by expanding $(\LRI{G}{}{a}{ii})^{-1}$ around $(\LRI{m}{}{a}{G})^{-1}$
\begin{EA}{ll}
  \frac{1}{\LRI{m}{}{a}{G}} -\frac{\LRI{G}{}{a}{ii}-\LRI{m}{}{a}{G}}{(\LRI{m}{}{a}{G})^2} +\OO{\varphi^{4Q_{\zeta}}\Psi^2}
  =&
  -w(1+\LRI{m}{}{a-1}{\Gc})+\frac{|z|^2}{1+\LRI{m}{}{a+}{G}} 
  \\
  &+w\left(\LRI{m}{}{a-1}{\Gc}-\LRI{m}{(i(a),\emptyset)}{a-1}{\Gc}\right)+\OO{\LRI{m}{}{a+1}{G}-\LRI{m}{(i(a),i(a))}{a+1}{G}} 
  \\
  &+w\LRI{\Zc}{}{a}{i}+\OO{\LRI{Z}{i(a)}{a}{i}}
  ,
\end{EA}
then taking the average over $i$ and expanding functions $(\LRI{m}{}{a}{G})^{-1}$ and $(1+\LRI{m}{}{a+1}{G})^{-1}$ around $m_c^{-1}$ and $(1+m_c)^{-1}$ respectively.
We end up with the following system of approximate linear equations with respect to $\Delta$ holding on $\tilde{S}_{\tilde{Q}_\zeta}$ with $\zeta$-high probability
\begin{EA}{lll}\EAy\label{eq:lin_sys_better_1}
    -\frac{\Delta_a}{m_c^2}+w\Delta'_{a-1}+\frac{|z|^2}{(1+m_c)^2}\Delta_{a+1}
    &=\,&
    \OO{\varphi^{4Q_{\zeta}}\Psi^2}+\left[w\LRI{\Zc}{}{a}{i}+\OO{\LRI{Z}{i(a)}{a}{i}}\right]
    \\
    &&+\left[w\left(\LRI{m}{}{a-1}{\Gc}-\LRI{m}{(i(a),\emptyset)}{a-1}{\Gc}\right)+\OO{\LRI{m}{}{a+1}{G}-\LRI{m}{(i(a),i(a))}{a+1}{G}}\right]
    ,\\ \EAy \label{eq:lin_sys_better_2}
    -\frac{\Delta'_a}{m_c^2}+w\Delta_{a+1}+\frac{|z|^2}{(1+m_c)^2}\Delta'_{a-1}
    &=&
    \OO{\varphi^{4Q_{\zeta}}\Psi^2}+\left[w\LRI{Z}{}{a}{i}+\OO{\LRI{\Zc}{i(a)}{a}{i}}\right]
    \\
    &&+\left[w\left(\LRI{m}{}{a+1}{G}-\LRI{m}{(\emptyset,i(a))}{a+1}{G}\right)+\OO{\LRI{m}{}{a-1}{\Gc}-\LRI{m}{(i(a),i(a))}{a-1}{\Gc}}\right]
    ,
\end{EA}
where by $[\cdot]$ we mean averaging over $i\in\{1,\ldots,N\}$.

We now give a heuristic argument of how the bound \eqref{eq:weak_concentration} can be improved.
Suppose that on a subset $\Sigma\subset\tilde{S}_{\tilde{Q}_{\zeta}}$ all the terms on the right-hand sides in the above system \eqref{eq:lin_sys_better_1}-\eqref{eq:lin_sys_better_2} are of order $\oo{\Psi}$ with high enough probability.
Then using the notation introduced in Section~\ref{sec:weak_concentration} we can rewrite the system \eqref{eq:lin_sys_better_1}-\eqref{eq:lin_sys_better_2} as
\begin{equation*}
  \Gamma \Delta
  =
  \oo{\Psi}
  .
\end{equation*}
From Lemma~\ref{pr:gamma_norm_bound} we can deduce that
\begin{equation}\label{eq:heuristic_1}
  \Lambda
  =
  \oo{\Psi}
  .
\end{equation}
If the estimate \eqref{eq:heuristic_1} holds with high enough probability, then we can use the union bound to show that it holds with high probability simultaneously on a $N^{-K}$-net in $\Sigma$ for $K>0$ big enough.
Then we can use continuity properties of $\Lambda$ to extend the estimate to the whole set $\Sigma$.
Therefore, to get a stronger estimate of $\Lambda$ it is enough to get a stronger estimate of the terms on the right-hand sides of \eqref{eq:lin_sys_better_1}-\eqref{eq:lin_sys_better_2}.
Next lemma shows which of these terms can be bounded by $\oo{\Psi}$.

\begin{lem}\label{pr:minor_differences_better}
  Let $\zeta>0$. Then with $\zeta$-high probability
  \begin{EA}{c}\EAy\label{eq:minor_differences_better_1}
    |\LRI{m}{(i(a),\emptyset)}{a-1}{\Gc}-\LRI{m}{}{a-1}{\Gc}|
    =
    \OO{\varphi^{5Q_{\zeta}}\frac{\Psi^2}{|w|}}
    ,\quad
    |\LRI{m}{(i(a),i(a))}{a+1}{G}-\LRI{m}{}{a+1}{G}|
    =
    \OO{\varphi^{5Q_{\zeta}}\Psi^2}
    ,
    \\\EAy\label{eq:minor_differences_better_2}
    |\LRI{m}{(\emptyset,i(a))}{a+1}{G}-\LRI{m}{}{a+1}{G}|
    =
    \OO{\varphi^{5Q_{\zeta}}\frac{\Psi^2}{|w|}}
    ,\quad
    |\LRI{m}{(i(a),i(a))}{a-1}{\Gc}-\LRI{m}{}{a-1}{\Gc}|
    =
    \OO{\varphi^{5Q_{\zeta}}\Psi^2}
    ,
  \end{EA}
for all $a\in\ZZ/n\ZZ$, $i\in\{1,\ldots,N\}$ and $w\in\tilde{S}_{\tilde{Q}_{\zeta}}$.
\end{lem}
\begin{proof}
By definition
\begin{equation*}
  \LRI{m}{(i(a),i(a))}{a+1}{G}-\LRI{m}{}{a+1}{G}
  =
  \frac{1}{N}\sum_{k=1}^{N}(\LRI{G}{(i(a),i(a))}{a+1}{kk}-\LRI{G}{}{a+1}{kk})
  .
\end{equation*}
  We show that for all $a\in\ZZ/n\ZZ$, $i\in\{1,\ldots,N\}$ with $\zeta$-high probability
  \begin{equation*}
    \LRI{G}{(i(a),i(a))}{a+1}{kk}-\LRI{G}{(i(a),\emptyset)}{a+1}{kk}
    =
    \OO{\varphi^{4Q_{\zeta}}\Psi^2}.
  \end{equation*}
From \eqref{eq:minor_differences_2}
\begin{equation*}
  \LRI{G}{(i(a),i(a))}{a+1}{kk}-\LRI{G}{(i(a),\emptyset)}{a+1}{kk}
  =
  \frac{\left[\left(G^{(i(a),i(a))}y_{i(a)}^*\right) \left( y_{i(a)} G^{(i(a),i(a))} \right)\right]_{k(a+1)k(a+1)}}{1+y_{i(a)} G^{(i(a),i(a))} y_{i(a)}^{*}}.
\end{equation*}
Since with $\zeta$-high probability $y_{i(a)} G^{(i(a),i(a))} y_{i(a)}^{*}=\LRI{m}{(i(a),i(a))}{a+1}{G}+\OO{\varphi^{Q_{\zeta}}\Psi}$, the absolute value of the denominator is bounded from below.
Consider the numerator of the above equation
\begin{EA}{lcl}
  \left[\left(G^{(i(a),i(a))}y_{i(a)}^*\right) \left( y_{i(a)} G^{(i(a),i(a))} \right)\right]_{k(a+1)k(a+1)}
  &=&
  \sum_{j,l=1}^{nN} G^{(i(a),i(a))}_{k(a+1)j}\xo{y}_{i(a)j} y_{i(a)l} G^{(i(a),i(a))}_{lk(a+1)}
  \\
  &=&\sum_{j,l=1}^{N} \LRI{G}{(i(a),i(a))}{a+1}{kj}\xo{x}_{i(a)j(a+1)} x_{i(a)l(a+1)} \LRI{G}{(i(a),i(a))}{a+1}{lk}
  .
\end{EA}
From \eqref{eq:lde_thm_2}-\eqref{eq:lde_thm_3} with $\zeta$-high probability
\begin{equation*}
  \sum_{l=1}^{N} \LRI{G}{(i(a),i(a))}{a+1}{kj}\xo{x}_{i(a)j(a+1)} x_{i(a)l(a+1)} \LRI{G}{(i(a),i(a))}{a+1}{lk}
  \leq
  \frac{\varphi^{Q_{\zeta}}}{N}\sqrt{\sum_{j,l=1}^{N} \left|\LRI{G}{(i(a),i(a))}{a+1}{kj} \LRI{G}{(i(a),i(a))}{a+1}{lk}\right|^2}.
\end{equation*}
Using Corollary~\ref{cor:weak_concentration_entries} we can show that 
\begin{equation*}
  \sum_{j,l=1}^{N} \left|\LRI{G}{(i(a),i(a))}{a+1}{kj} \LRI{G}{(i(a),i(a))}{a+1}{lk}\right|^2
  =
  \OO{N^2 \varphi^{8Q_{\zeta}}\Psi^4} + \OO{N \varphi^{4Q_{\zeta}}\Psi^2} + \OO{1}.
\end{equation*}
with $\zeta$-high probability.
Therefore 
\begin{equation*}
    \frac{\varphi^{Q_{\zeta}}}{N}\sqrt{\OO{N^2 \varphi^{8Q_{\zeta}}\Psi^4} + \OO{N \varphi^{4Q_{\zeta}}\Psi^2} + \OO{1}}
    =
    \OO{\varphi^{5Q_{\zeta}}\Psi^2}.
\end{equation*}
From \eqref{eq:minor_differences_1} and Corollary~\ref{cor:weak_concentration_entries}
\begin{equation*}
  \LRI{G}{}{a+1}{kk}-\LRI{G}{(i(a),\emptyset)}{a+1}{kk}
  =
  \frac{G_{i(a)k(a+1)}G_{k(a+1)i(a)}}{G_{i(a)i(a)}}
  =
  \OO{\varphi^{4Q_{\zeta}}\Psi^2}
\end{equation*}
and so the first estimate is proven. The other estimates can be shown using a similar argument.
\end{proof}

We now want to improve the estimates of the terms $|w[\LRI{\Zc}{}{a}{i}]|$, $|[\LRI{Z}{i(a)}{a}{i}]|$, $|w[\LRI{Z}{}{a}{i}]|$ and $|[\LRI{\Zc}{i(a)}{a}{i}]|$. 
As in \cite{BourYauYin} we use Proposition~\ref{thm:abs_dec} (Abstract decoupling lemma from \cite{PillYin}) to bound these terms.
However, in the Abstract decoupling lemma we need to work with deterministic estimates, therefore the $\oo{\Psi}$ bound cannot be achieved using this method.
Due to this restriction, we introduce the following deterministic parameters
\begin{equation*}
  \tilde{\Lambda}
  :=
  \varphi^{2Q_1}\left(\frac{1}{\sqrt{N}}+\frac{1}{N\eta}\right)
  ,\quad
  \tilde{\Psi}
  :=
  \varphi^{2Q_1}\left(\frac{1}{\sqrt{N}}+\frac{1}{N\eta}+\frac{1}{\sqrt{\eta}N^{3/4}}\right).
\end{equation*}
By Theorem~\ref{thm:continuity_argument} we know that on $\tilde{S}_{\tilde{Q}_{1}}$ with $1$-high probability 
\begin{equation}\label{eq:determ_param_bounds}
  \Lambda
  =
  \OO{\tilde{\Lambda}}
  \quad
  \mbox{and}
  \quad
  \Psi
  =
  \OO{\tilde{\Psi}}
.  
\end{equation}
Thus, Theorem~\ref{thm:concentration_main} can be deduced from the following proposition.

\begin{lem}\label{pr:strong_abs_dec}
  For any $D>0$, for any $w\in\tilde{S}_{\tilde{Q}_{1}}\cap \{\eta\geq N^{-1/2}\}$  with probability at least $1-N^{-D}$
\begin{equation}\label{eq:concentration_main_1}
  \max_a \left|\frac{1}{N} \sum_{i=1}^{N} w\LRI{\Zc}{}{a}{i}\right|+\left|\frac{1}{N}\sum_{i=1}^{N} \LRI{Z}{(i)}{a}{i}\right|
  =
  \OO{\frac{\tilde{\Psi}}{N^{1/10}}}.
\end{equation}
\end{lem}
\begin{proof}[Proof of Theorem~\ref{thm:concentration_main}]
  Consider the system of approximate equations \eqref{eq:lin_sys_better_1}-\eqref{eq:lin_sys_better_2} on $\tilde{S}_{\tilde{Q}_1}\cap\{\eta\geq N^{-1/2}\}$.
Using the Lemmas~\ref{pr:minor_differences_better} and \ref{pr:strong_abs_dec} as well as the notation of Section~\ref{sec:weak_concentration}, we see that for  $w\in\tilde{S}_{\tilde{Q}_1}\cap\{\eta\geq N^{-1/2}\}$ the system
  \begin{equation*}
    \Gamma \Delta
    =
    \OO{\varphi^{4Q_1}\Psi^2}+\OO{\frac{\tilde{\Psi}}{N^{1/10}}}+\OO{\varphi^{5Q_1}\Psi^2}
  \end{equation*}
holds with probability at least $1-N^{-D}$ for $N$ big enough.
Note that we used the fact that $e^{-\varphi}N^D\rightarrow 0,\quad N\rightarrow \infty$.
From \eqref{eq:determ_param_bounds} and the definition of $\tilde{\Psi}$ for $w\in\tilde{S}_{\tilde{Q}_1}\cap\{\eta\geq N^{-1/2}\}$ we have that
\begin{equation*}
  \varphi^{5Q_1}\Psi^2
  =
  \oo{\frac{1}{\sqrt{N}}+\frac{1}{N\eta}+\frac{1}{\sqrt{\eta}N^{3/4}}}
  ,\quad
  \OO{\frac{\tilde{\Psi}}{N^{1/10}}}
  =
  \oo{\frac{1}{\sqrt{N}}+\frac{1}{N\eta}+\frac{1}{\sqrt{\eta}N^{3/4}}}
  .
\end{equation*}
Thus, due to Lemma~\ref{pr:gamma_norm_bound} for any $w\in\tilde{S}_{\tilde{Q}_1}\cap\{\eta\geq N^{-1/2}\}$ with probability at least $1-N^{-D}$
\begin{equation}\label{eq:strong_bound_1}
  \Lambda
  =
  \oo{\frac{1}{\sqrt{N}}+\frac{1}{N\eta}+\frac{1}{\sqrt{\eta}N^{3/4}}}
  .
\end{equation}
If we take $D\geq 8$, then using the union bound we can show that \eqref{eq:strong_bound_1} holds for a $N^{-3}$-net in $\tilde{S}_{\tilde{Q}_1}\cap\{\eta\geq N^{-1/2}\}$ with probability at least $1-N^{6-D}$.
By the continuity properties of $\Lambda$ stated in Lemma~\ref{lem:derivative_estimates} we can extend \eqref{eq:strong_bound_1} to the whole set $\tilde{S}_{\tilde{Q}_1}\cap\{\eta\geq N^{-1/2}\}$. Therefore, Theorem~\ref{thm:concentration_main} is proven.
\end{proof}
\begin{proof}[Proof of Lemma~\ref{pr:strong_abs_dec}]
We shall prove the bound for 
\begin{equation*}
  \left|\frac{1}{N}\sum_{i=1}^{N} \LRI{Z}{(i)}{a}{i}\right|,
\end{equation*}
and for the other terms the proof is similar.

Note that
\begin{equation*}
  \LRI{Z}{(i)}{a}{i}
  =
  (1-\EE_{y_{i(a)}})\left[\frac{1}{w\LRI{\Gc}{(i(a),\emptyset)}{a}{ii}}\right]
  .
\end{equation*}

Fix $a\in\ZZ/n\ZZ$. Following the notation used in the Proposition~\ref{thm:abs_dec}, we can take
\begin{equation*}
  \Ic:=\{1,\ldots,nN\}
  ,\quad
  \Ic_i:=\{i(a)\}
  ,\quad
  1\leq i\leq N
  ,
\end{equation*}
and 
\begin{equation*}
  \QQ_i
  =
  1-\EE_{y_{i(a)}}.
\end{equation*}

Suppose that the variables 
\begin{equation*}
  \frac{1}{w\LRI{\Gc}{(i(a),\emptyset)}{a}{ii}}
\end{equation*}
satisfy the conditions $(i)-(iii)$ of the Abstract decoupling lemma with 
\begin{equation*}
  \mathcal{X}
  =
  \frac{\tilde{\Psi}}{|w|^{1/4}}
  ,\quad
  \mathcal{Y}
  =
  \frac{1}{|w|^{1/4}}.
\end{equation*}

Then by Chebyshev inequality
\begin{equation*}
  \Pr{\left|\frac{1}{N}\sum_{i=1}^{N} \LRI{Z}{(i)}{a}{i}\right|\geq (Cp)^4\left(\frac{\tilde{\Psi}}{N^{1/10}}\right)}
  \leq
  \frac{(Cp)^{4p}\left(\frac{\tilde{\Psi}^2}{|w|^{1/2}}+\frac{1}{N|w|^{1/4}}\right)^p}{(Cp)^{4p}\left(\frac{\tilde{\Psi}}{N^{1/10}}\right)^p}
  .
\end{equation*}
We now estimate the RHS of the inequality. For $\{w\in\CC_+\,:\, \eta\geq N^{-1/2}\}$, we have 
\begin{equation*}
  \frac{N^{1/10}\tilde{\Psi}}{|w|^{1/2}}
  \leq
  N^{1/10}\varphi^{Q_{1}}\left(\frac{1}{\sqrt{N\eta}}+\frac{1}{N\eta^{3/2}}+\frac{1}{N^{3/4}\eta}\right)
  =
  \OO{\frac{N^{1/10}\varphi^{Q_{1}}}{N^{1/4}}}
  =
  \OO{\frac{1}{N^{1/10}}}
  .
\end{equation*}
Similarly we can obtain the estimate for the other term
\begin{equation*}
  \frac{N^{1/10}}{N|w|^{1/4}\tilde{\Psi}}
  \leq
  N^{1/10}\frac{1}{\sqrt{N}\eta^{1/4}+\eta^{-3/4}+N^{1/4}\eta^{-1/4}}  
  =
  \OO{\frac{1}{N^{1/10}}}.
\end{equation*}
Then the proposition is proven if we take $p=\lfloor10D\rfloor+1$.
Therefore it is enough to show that the conditions $(i)-(iii)$ of the Proposition~\ref{thm:abs_dec} are satisfied by the variables
\begin{equation*}
  \left\{\frac{1}{w\LRI{\Gc}{(i(a),\emptyset)}{a}{ii}}\right\}_{i=1}^{N}.
\end{equation*}

Firstly, we note that the uniform subexponential decay condition allows us to truncate the entries and thus we can easily show that with $\zeta$-high probability the condition $(ii)$ is satisfied.

Let $\Xi\subset \Omega$ be the set on which Lemma~\ref{lem:lde} and Corollary~\ref{cor:weak_concentration_entries} hold for all sets $\TT,\UU\subset \{1,\ldots,N\}$ with $|\TT|+|\UU|\leq 2p$ and the entries of the matrices $X_a$ are bounded by $N^C$. 
This set is of $\zeta$-high probability.

We  now verify that condition $(i)$ is satisfied on $\Xi$. 
Let $\AA\subset \{1,\ldots,N\},i\in \AA$ and suppose that on $\Xi$
\begin{equation}\label{eq:2}
  \frac{1}{w\LRI{G}{(i(a),\emptyset)}{a}{ii}}
  =
  S_{i,\AA}^{(1)}+S_{i,\AA}^{(2)}
\end{equation}
with $S_{i,\AA}^{(1)}=\OO{(C|\AA|\tilde{\Psi}|w|^{-1/4})^{|\AA|}}, S_{i,\AA}^{(2)}=\OO{N^{C|\AA|}}$ and $\QQ_{\AA}S_{i,\AA}^{(2)}=0$.
We now show that with this decomposition we can obtain decomposition in $(i)$.

Let $\Xi^*\subset\Omega$ be set a containing $\Xi$ and independent of $A$ (for example, cylindrical set $\pi^{-1}_{\Omega_A^c}\pi_{\Omega_A^c}\Xi$). Define
\begin{equation*}
  S_{i,\AA}
  :=
  \QQ_{\AA}1_{\Xi}S_{i,\AA}^{(1)}
  ,\quad
  \tilde{S}_{i,\AA}
  :=
  \frac{1}{w\LRI{G}{(i(a),\emptyset)}{a}{ii}} - 1_{\Xi^*\setminus \Xi} S_{i,\AA}^{(2)}
  .
\end{equation*}
Then
\begin{equation*}
  S_{i,\AA}+\QQ_{\AA}1_{\Xi^{\complement}}\tilde{S}_{i,\AA}
  =
  \QQ_{\AA}1_{\Xi}S_{i,\AA}^{(1)}+\QQ_{\AA}1_{\Xi^{\complement}}\left(\frac{1}{w\LRI{G}{(i(a),\emptyset)}{a}{ii}} - 1_{\Xi^*\setminus \Xi} S_{i,\AA}^{(2)}\right)
  =
  \QQ_{\AA}\frac{1}{w\LRI{G}{(i(a),\emptyset)}{a}{ii}}-\QQ_{\AA}1_{\Xi^*}S_{i,\AA}^{(2)}
  =
  \QQ_{\AA}\frac{1}{w\LRI{G}{(i(a),\emptyset)}{a}{ii}}
  .
\end{equation*}

The last thing to show is that in $\Xi$ decomposition \eqref{eq:2} holds for all $\AA\subset\{1,\ldots,N\}$ with $|\AA|\leq p$.
We begin with the cases $|\AA|=1$ and $|\AA|=2$. 
First of all, note that
\begin{equation*}
  \QQ_i \frac{1}{w\LRI{\Gc}{(i(a),\emptyset)}{a}{ii}}
  =
  \LRI{Z}{(i)}{a}{i}
  =
  \OO{\tilde{\Psi}},
\end{equation*}
so in this case 
\begin{equation*}
  \frac{1}{w\LRI{\Gc}{(i(a),\emptyset)}{a}{ii}}
  =
  \QQ_i\left[ \frac{1}{w\LRI{\Gc}{(i(a),\emptyset)}{a}{ii}}\right] +\E_i\left[ \frac{1}{w\LRI{\Gc}{(i(a),\emptyset)}{a}{ii}}\right]
\end{equation*}
and \eqref{eq:2} holds.

Now let $\AA=\{i,j\}$ and let
\begin{equation*}
  A=
  \begin{pmatrix}
    y_{i(a)}
    \\
    y_{j(a)}
  \end{pmatrix}.
\end{equation*}
Then by Schur complement formula we have that
  \begin{equation*}
    \LRI{\Gc}{(i(a),\emptyset)}{a}{kl}
    =
    \left(\frac{1}{-w(I+A\LRI{G}{(i(a),\AA)}{}{}A^*)}\right)_{kl}
    ,\quad
    k,l\in\{i,j\}.
  \end{equation*}
After calculating directly the $ii$th element of the inverse of the matrix we get
\begin{equation*}
  -\frac{1}{w\LRI{\Gc}{(i(a),\emptyset)}{a}{ii}}
  =
  \left(1+y_{i(a)}\LRI{G}{(i(a),\AA)}{}{}y_{i(a)}^*\right)
  - \frac{\left(y_{i(a)}\LRI{G}{(i(a),\AA)}{}{}y_{j(a)}^*\right)\left(y_{j(a)}\LRI{G}{(i(a),\AA)}{}{}y_{i(a)}^*\right)}{1+y_{j(a)}\LRI{G}{(i(a),\AA)}{}{}y_{j(a)}^*}
  .
\end{equation*}
Introducing the following notation
\begin{equation*}
    R_{ii}
    :=
    1+y_{i(a)}\LRI{G}{(i(a),A)}{}{}y_{i(a)}^*
    ,\quad
    R_{ij}
    :=
    y_{i(a)}\LRI{G}{(i(a),A)}{}{}y_{j(a)}^*
    ,\quad
    R_{ji}
    :=
    y_{j(a)}\LRI{G}{(i(a),A)}{}{}y_{i(a)}^*
\end{equation*}
we can rewrite the above equality as
\begin{equation}\label{eq:1}
  -\frac{1}{w\LRI{\Gc}{(i(a),\emptyset)}{a}{ii}}
  =
  R_{ii}-\frac{R_{ij} R_{ji}}{R_{jj}}.
\end{equation}
Using the estimates from the Corollary~\ref{cor:weak_concentration_entries} we get approximated values of the terms on the RHS of \eqref{eq:1}
\begin{equation*}
  1+y_{j(a)}\LRI{G}{(i(a),\AA)}{}{}y_{j(a)}^*
  =
  -\frac{1}{w m_c}+\OO{\frac{\tilde{\Psi}}{w}}
  ,
\end{equation*}
so that
\begin{equation*}
  \frac{1}{R_{ii}}
  =
  -wm_c+\OO{w\tilde{\Psi}}.
\end{equation*}
The terms $R_{ij}$ and $R_{ji}$ can be decomposed in the following way
\begin{EA}{c}
  R_{ij}
  =
  \sum_{k,l=1}^{N} \xo{x}_{i(a),k(a+1)} \LRI{G}{(i(a),\AA)}{a+1}{kl} x_{j(a),l(a+1)} -\xo{z}\sum_{k=1}^N \xo{x}_{i(a),k(a+1)}\LRI{G}{(i(a),\AA)}{a+1}{kj}
  =:
  R_{ij}^{(1)}+R_{ij}^{(2)}
  ,\\
  R_{ji}
  =
  \sum_{k,l=1}^{N} \xo{x}_{j(a),k(a+1)} \LRI{G}{(i(a),\AA)}{a+1}{kl} x_{i(a),l(a+1)} -z\sum_{k=1}^N x_{i(a),l(a+1)}\LRI{G}{(i(a),\AA)}{a+1}{kj}
  =:
  R_{ji}^{(1)}+R_{ji}^{(2)}
\end{EA}
and we observe that the terms $R_{ij}^{(2)}$ and $R_{ji}^{(2)}$ do not depend on $y_{j(a)}$. 

Therefore we have that
\begin{equation*}
  -\frac{1}{w\LRI{\Gc}{(i(a),\emptyset)}{a}{ii}}
  =
  R_{ii}-(R_{ij}^{(1)}+R_{ij}^{(2)})( R_{ji}^{(1)}+R_{ji}^{(2)})(-wm_c+\OO{w\tilde{\Psi}})
\end{equation*}
and
\begin{equation*}
  \QQ_j(R_{ii}+wm_cR_{ij}^{(2)}R_{ji}^{(2)})
  =
  0.
\end{equation*}
Using again Lemma~\ref{lem:lde_byy}, we obtain the following estimates
\begin{EA}{c}
  R_{ij}^{(1)}
  =
  \OO{\tilde{\Psi}}
  ,\quad
  R_{ji}^{(1)}
  =
  \OO{\tilde{\Psi}}
  ,\quad
  R_{ij}^{(2)}
  =
  \OO{\frac{\tilde{\Psi}}{|w|}}
  ,\quad
  R_{ji}^{(2)}
  =
  \OO{\frac{\tilde{\Psi}}{|w|}}
  .
\end{EA}
Note that if $\eta>N^{-1/2}$, then $\tilde{\Psi}|w|^{-1/2}=\oo{1}$. Using this fact we see that
\begin{equation*}
    -\frac{1}{w\LRI{\Gc}{(i(a),\emptyset)}{a}{ii}}-R_{ii} - wm_c R_{ij}^{(2)}R_{ji}^{(2)}
    =
    \OO{\tilde{\Psi}^2+\frac{\tilde{\Psi}^3}{|w|}}
    =
    \OO{\frac{\tilde{\Psi}^2}{\sqrt{|w|}}\left(\sqrt{|w|}+\frac{\tilde{\Psi}}{\sqrt{|w|}}\right)}
    =
    \OO{\frac{\tilde{\Psi}^2}{|w|^{1/2}}}
    ,
\end{equation*}
and so the case $|\AA|=2$ is verified.

Now we consider general $\AA\subset\{1(a),\ldots,N(a)\}$ with $i\in\AA$.  
Define matrices $A=(y_{k(a)})_{k\in\AA}\in \mathcal{M}_{|\AA|\times nN}(\CC)$ and $\tilde{A}=(y_{k(a)})_{k\in\AA\setminus i}\in\mathcal{M}_{|\AA|-1\times nN}(\CC)$. 
By the Schur complement formula we have
\begin{equation*}
  \left(\LRI{\Gc}{(i(a),\emptyset)}{a}{kl}\right)_{k,l\in\AA}
  =
  \left(\frac{1}{-w(I+A\LRI{G}{(i(a),\AA)}{}{}A^*)}\right)
  .
\end{equation*}
Similarly to the case $|\AA|=2$, we introduce notation
  \begin{equation*}
    R_{ii}
    :=
    1+y_{i(a)}\LRI{G}{(i(a),A)}{}{}y_{i(a)}^*
    ,\quad
    R_{ij}
    :=
    y_{i(a)}\LRI{G}{(i(a),A)}{}{}y_{j(a)}^*
    ,\quad
    R_{ji}
    :=
    y_{j(a)}\LRI{G}{(i(a),A)}{}{}y_{i(a)}^*
    ,\quad
    j\in A\setminus\{i\}.
  \end{equation*}
Using again Schur complement formula we have
  \begin{equation*}
    -\frac{1}{w\LRI{\Gc}{(i(a),\emptyset)}{a}{ii}}
    =
    R_{ii} - w\sum_{k,l\in \AA\setminus \{i\}} R_{ik} \left(\frac{1}{-w(I+\tilde{A}\LRI{G}{(i(a),\AA)}{}{}\tilde{A}^*)}\right)_{kl}  R_{li}
    .
  \end{equation*}
Define the matrix $R=(R_{kl})_{k,l\in\AA\setminus\{i\}}$ by
  \begin{equation*}
    I+\tilde{A}\LRI{G}{(i(a),\AA)}{}{}\tilde{A}^*
    =
    \bfr(I+R)
    ,
  \end{equation*}
 where
  \begin{equation*}
    \bfr
    :=
    1+\LRI{m}{(i(a),A)}{a+1}{G}-\frac{|z|^2}{w(1+\LRI{m}{(i(a),A)}{a-1}{\Gc})}
    \sim
    \frac{1}{|w|}.
  \end{equation*}
From Lemma~\ref{lem:lde}, \eqref{eq:weak_concentration_entries_1} and Theorem~\ref{thm:continuity_argument}, the entries of the matrix $\bfr R=I+\tilde{A}\LRI{G}{(i(a),\AA)}{}{}\tilde{A}^*-\bfr I$ are of size $\OO{\tilde{\Psi}|w|^{-1}}$, which implies that
\begin{EA}{c}
  R_{kl}
  =
  \OO{\tilde{\Psi}}
  ,\quad
  k,l\in \AA\setminus \{i\}
  ,\\
  R_{ik}
  =
  \OO{\frac{\tilde{\Psi}}{|w|}}
  ,\quad
  R_{ki}
  =
  \OO{\frac{\tilde{\Psi}}{|w|}}
  ,\quad
  k\in \AA\setminus \{i\}
  .
\end{EA}
Then we have
  \begin{equation*}
    -\frac{1}{w\LRI{\Gc}{(i(a),\emptyset)}{a}{ii}}
    =
    R_{ii} - w\sum_{k,l\in A\setminus i} R_{ik} \frac{1}{-wb}\left(I+\sum_{k=1}^{|\AA|-3}(-R)^k +(-R)^{|\AA|-2}(I+R)^{-1}\right)_{kl} R_{li}
    .
  \end{equation*}
Parameter $\bfr$ is independent of $A$ and each $R_{kl}$ is independent of $\{y_{j(a)}\, : \, j\in\AA\setminus\{k,l\}\}$.
Thus
  \begin{equation*}
    \QQ_{\AA}\left[ R_{ii} - w\sum_{k,l\in A\setminus i} R_{ik} \frac{1}{-wb}\left(I+\sum_{k=1}^{|\AA|-3}(-R)^k \right)_{kl} R_{li}\right]
    =
    0
    .
  \end{equation*}
Note that
  \begin{equation*}
    \frac{1}{-wb}(I+R)^{-1}
    =
    \left(\LRI{\Gc}{(i(a),i(a))}{a}{kl}\right)_{k,l\in A\setminus i}
  \end{equation*}
and thus 
  \begin{equation*}
     w\sum_{k,l\in A\setminus i} R_{ik} \frac{1}{-wb}\left((-R)^{|\AA|-2}(I+R)^{-1}\right)_{kl} R_{li}
     =
     \OO{\frac{1}{|w|^{1/4}}\left(\frac{|\AA|\tilde{\Psi}}{|w|^{1/4}}\right)^{|\AA|}},
  \end{equation*}
which finishes the proof.
\end{proof}

\begin{rem}
  If in the above proof we choose appropriately $p=p_N$ instead of fixed $p$, we can obtain a stronger bound in \eqref{eq:concentration_main_1}.
For example, we can show that for $\zeta>0$
\begin{equation*}
  \max_a \left|\frac{1}{N} \sum_{i=1}^{N} w\LRI{\Zc}{}{a}{i}\right|+\left|\frac{1}{N}\sum_{i=1}^{N} \LRI{Z}{(i)}{a}{i}\right|
  =
  \OO{\frac{\tilde{\Psi}^2}{\eta^{1/2}}+\frac{1}{N\eta^{1/4}}}
\end{equation*}
holds with $\zeta$-high probability.
\end{rem}


\paragraph{Acknowledgements.}
I would like to thank my PhD advisors Mireille Capitaine and Michel Ledoux for introducing the problem to me, fruitful discussions and reading the manuscript.

\appendix

\subsection*{Appendix: Proof of Lemma~\ref{lem:gaus_diff}}

  The proof relies on the decoupling formula first introduced in \cite{KhorKhorPast} (see also \cite{PastShch}, chapters 18 and 19) which we state below before proceeding to the proof. 

\begin{pr}\emph{(\cite{PastShch}, Proposition 18.1.4)}\label{pr:gaus_diff}
  Let $V$ be a random variable such that $\EE[|V|^{p+2}]<\infty$ for a certain nonnegative integer $p$.
  Then for any function $\Phi\, : \,\RR  \rightarrow \CC$ of the class $\mathcal{C}^{p+1}$ with bounded derivatives $\Phi^{(l)}, l=1,\ldots,p+1$, we have
  \begin{equation}
    \label{eq:gaus_diff_formula}
    \EE[V\Phi(V)]
    =
    \sum_{l=0}^p \frac{\kappa_{l+1}}{l!}\EE[\Phi^{(l)}(V)]+\epsilon_p
    ,
  \end{equation}
  where the remainder term $\epsilon_p$ admits the bound
  \begin{equation}
    \label{eq:gaus_diff_error}
    |\epsilon_p|
    \leq
    C_p \EE[|V|^{p+2}]\sup_{t\in\RR}|\Phi^{(p+1)}(t)|
  \end{equation}
  for a constant $C_p$ depending on $p$.
\end{pr}

\begin{proof}[Proof of Lemma~\ref{lem:gaus_diff}]
  Define an interpolation matrix $X(s)=\sqrt{s}X+\sqrt{1-s}\hat{X}$ and the resolvent matrices $G^s(z,w)=((Y_z(s))^*Y_z(s)-w)^{-1}$ and $\Gc^s(z,w)=(Y_z(s)(Y_z(s)^*-w)^{-1}$, where  $Y_z(s)=X(s)-z$.
  Note that $G^0(z,w)=\hat{G}(z,w)$ and $G^{1}(z,w)=G(z,w)$.
  Denote the normalised partial traces of the matrices $G^s(z,w)$ and $\Gc^s(z,w)$ by $\LRI{m}{}{a}{G}(z,w,s)$ and  $\LRI{m}{}{a}{\Gc}(z,w,s)$.
  Then
  \begin{EA}{cl}
    \EE[\LRI{m}{}{a}{G}(z,w)-\LRI{m}{}{a}{\hat{G}}(z,w)]
    &=
    \EE[\LRI{m}{}{a}{G}(z,w,1)-\LRI{m}{}{a}{G}(z,w,0)]
    \\
    &=
    \EE[\int_0^1 \frac{\partial}{\partial s} \LRI{m}{}{a}{G}(z,w,s) ds]
    \\\EAy\label{eq:gaus_diff_1}
    &=
    \frac{1}{N}\sum_{i=1}^{N}\EE[\int_0^1 \frac{\partial}{\partial s} \LRI{G}{}{a}{ii}(z,w,s) ds]
  \end{EA}

  Denote by $\LRI{D}{}{b,b+1}{kl}$ the derivation with respect to $\LRI{x}{s}{b,b+1}{kl}$.
  Then from the definition of the matrices $G^s(z,w)$ and $\Gc^s(z,w)$ we have
  \begin{EA}{rl}
    \frac{\partial}{\partial s} \LRI{G}{}{a}{ii}(s)
    &=
    \sum_{b=1}^{n}\sum_{k,l=1}^{N} \LRI{D}{}{b,b+1}{kl} \LRI{G}{}{a}{ii}(s) \cdot \frac{\partial \LRI{x}{s}{b,b+1}{kl}}{\partial s}
    \\\EAy\label{eq:gaus_diff_2}
    &=
    \sum_{b=1}^{n}\sum_{k,l=1}^{N} \LRI{D}{}{b,b+1}{kl} \LRI{G}{}{a}{ii}(s) \cdot (\frac{1}{2\sqrt{s}} \LRI{x}{}{b,b+1}{kl} - \frac{1}{2\sqrt{1-s}} \LRI{\hat{x}}{}{b,b+1}{kl})
    \end{EA}
    Now we shall apply the Proposition~\ref{pr:gaus_diff} to the functions $\Phi_1$ and $\Phi_0$ defined by
    \begin{equation*}
      \Phi_1(\LRI{x}{}{b,b+1}{kl})
      :=
      \frac{1}{\sqrt{s}}\sum_{i=1}^{N} \LRI{D}{}{b,b+1}{kl} \LRI{G}{}{a}{ii}(s)
      ,
      \quad
      \Phi_0(\LRI{\hat{x}}{}{b,b+1}{kl})
      :=
      \frac{1}{\sqrt{1-s}}\sum_{i=1}^{N} \LRI{D}{}{b,b+1}{kl} \LRI{G}{}{a}{ii}(s)
    \end{equation*}
    From the differentiation formula \eqref{eq:gaus_diff_formula} we have that
    \begin{EA}{ll}
      \EE[\Phi_1(\LRI{x}{}{b,b+1}{kl})\LRI{x}{}{b,b+1}{kl}]
      =
      \kappa_1 \EE[\Phi_1(\LRI{x}{}{b,b+1}{kl})] + \kappa_2 \EE[\frac{\partial \Phi_1(\LRI{x}{}{b,b+1}{kl})}{\partial \LRI{x}{}{b,b+1}{kl}}] + \LRI{\epsilon}{1}{b}{kl}
      \\
      \EE[ \Phi_0(\LRI{\hat{x}}{}{b,b+1}{kl})\LRI{\hat{x}}{}{b,b+1}{kl}]
      =
      \hat{\kappa}_1 \EE[ \Phi_0(\LRI{\hat{x}}{}{b,b+1}{kl})] + \hat{\kappa}_2 \EE[ \frac{\partial \Phi_0(\LRI{\hat{x}}{}{b,b+1}{kl})}{\partial \LRI{\hat{x}}{}{b,b+1}{kl}}] + \LRI{\epsilon}{0}{b}{kl}
      ,
    \end{EA}
    where
    \begin{equation*}
      \LRI{\epsilon}{1}{b}{kl}
      \leq
      C \EE[|\LRI{x}{}{b,b+1}{kl}|^3] \sup_{t\in\RR} |\frac{\partial^2 \Phi_1}{\partial (\LRI{x}{}{b,b+1}{kl})^2 }(t)|
      ,\quad
      \LRI{\epsilon}{0}{b}{kl}
      \leq
      C \EE[|\LRI{\hat{x}}{}{b,b+1}{kl}|^3] \sup_{t\in\RR} |\frac{\partial^2 \Phi_0}{\partial (\LRI{\hat{x}}{}{b,b+1}{kl})^2 }(t)|
      .
    \end{equation*}
    The first two moments of the entries of $X$ and $\hat{X}$ are equal, therefore $\kappa_1=\hat{\kappa}_1=0$ and $\kappa_2=\hat{\kappa}_2$.
    Moreover, from the definition of the functions $\Phi_1$ and $\Phi_0$ we have that
    \begin{equation*}
      \frac{\partial \Phi_1(\LRI{x}{}{b,b+1}{kl})}{\partial \LRI{x}{}{b,b+1}{kl}}
      =
      \frac{\partial \Phi_0(\LRI{\hat{x}}{}{b,b+1}{kl})}{\partial \LRI{\hat{x}}{}{b,b+1}{kl}}
    \end{equation*}
    Therefore, 
    \begin{equation*}
      |\EE[\Phi_1(\LRI{x}{}{b,b+1}{kl})\LRI{x}{}{b,b+1}{kl}]-\EE[ \Phi_0(\LRI{\hat{x}}{}{b,b+1}{kl})\LRI{\hat{x}}{}{b,b+1}{kl}]|
      \leq
      |\LRI{\epsilon}{1}{b}{kl}|+|\LRI{\epsilon}{0}{b}{kl}|
      .
    \end{equation*}
    Recall that by \eqref{eq:gaus_diff_1} and \eqref{eq:gaus_diff_2} 
    \begin{equation*}
      \EE[\LRI{m}{}{a}{G}(z,w)-\LRI{m}{}{a}{\hat{G}}(z,w)]
      =
      \frac{1}{N} \int_0^1 ds \sum_{b=1}^n \sum_{k,l=1}^N (\EE[\Phi_1(\LRI{x}{}{b,b+1}{kl})\LRI{x}{}{b,b+1}{kl}] -\EE[ \Phi_0(\LRI{\hat{x}}{}{b,b+1}{kl})\LRI{\hat{x}}{}{b,b+1}{kl}])
    \end{equation*}
    and thus
    \begin{equation*}
      |\EE[\LRI{m}{}{a}{G}(z,w)-\LRI{m}{}{a}{\hat{G}}(z,w)]|
      \leq
      \frac{1}{N} \int_{0}^1 ds \sum_{b=1}^n \sum_{k,l=1}^N (|\LRI{\epsilon}{1}{b}{kl}|+|\LRI{\epsilon}{0}{b}{kl}|)
      .
    \end{equation*}
    Since
    \begin{equation*}
      \EE[|\LRI{x}{}{b,b+1}{kl}|^3]
      \sim
      \EE[|\LRI{\hat{x}}{}{b,b+1}{kl}|^3]
      \sim
      \frac{1}{N^{3/2}}
    \end{equation*}
    it will be enough to show that $\exists C>O$ such that for all $a,b\in \ZZ/n\ZZ$ and $k,l\in \llbracket 1,N\rrbracket$
    \begin{equation*}
      |\LRI{D}{3}{b,b+1}{kl} \sum_{i=1}^{N}\LRI{G}{}{a}{ii}|
      \leq
      C
      .
    \end{equation*}
    In order to obtain this bound, we differentiate the diagonal entries of the matrix $G$
    \begin{equation*}
      \LRI{D}{}{b,b+1}{kl} \LRI{G}{}{a}{ii}
      =
      - \LRI{G}{}{a,b+1}{il} \LRI{(Y_z G)}{}{ba}{ki} - \LRI{(GY^*_z)}{}{ab}{ik} \LRI{G}{}{b+1,a}{li}
      .
    \end{equation*}
    To identify the possible terms appearing in $\LRI{D}{3}{b}{kl} \sum_{i=1}^{N}\LRI{G}{}{a}{ii}$ we calculate the following derivatives
    \begin{EA}{ll}
      \LRI{D}{}{}{kl} \LRI{G}{}{}{ij}
      &=
      - \LRI{G}{}{}{il} \LRI{(Y_z G)}{}{}{ki} - \LRI{(GY^*_z)}{}{}{ik} \LRI{G}{}{}{lj}
      \\
      \LRI{D}{}{}{kl} \LRI{(Y_z G)}{}{}{ij}
      &=
      \delta_{ik}G_{lj}-(Y_z G)_{il}(Y_z G)_{kj} - (Y_z G Y_z^*)_{ik}G_{lj}
      \\
      \LRI{D}{}{}{kl} \LRI{(GY^*_z)}{}{}{ij}
      &=
      G_{il}\delta_{kj} - G_{il}(Y_z G Y^*)_{kj} - (G Y_z^*)_{ik}(G Y_z^*)_{lj}
      \\
      \LRI{D}{}{}{kl} \LRI{(Y_z GY^*_z)}{}{}{ij}
      &=
      \delta_{ik}(GY_z^*)_{lj} + (Y_z G)_{il}\delta_{kj} - (Y_z G Y_z^*)_{ik}(G Y_z^*)_{lj} - (Y_z G)_{il} (Y_z G Y_z^*)_{kj}
    \end{EA}
    We see that $\LRI{D}{3}{b}{kl} \LRI{G}{}{a}{ii}$ can be written as a finite sum of the terms $M^{(1)}_{iq_1}M^{(2)}_{q_2q_3}M^{(3)}_{q_4q_5}M^{(4)}_{q_6 i}$, where 
    \begin{equation*}
      M^{(p)}_{qq'}\in\{G_{qq'},(Y_z G)_{qq'}, (G Y^*_z)_{qq'}, (Y_z G Y_z^*)_{qq'}, \delta_{qq'}\}
    \end{equation*}
 and $q_1,q_2,q_3,q_4,q_5,q_6 \in \{k(b),l(b+1)\}$.
     Using this representation and the Cauchy-Schwarz inequality we conclude that it is sufficient to bound the spectral norms of the matrices $G$, $Y_z G$, $G Y^*_z$ and $Y_z G Y_z^*$.
     From the definition $G=(Y_z^* Y_z -w)^{-1}$ we see that 
     \begin{equation*}
       s^2_N(G) 
       = 
       \max_{t\in \mathrm{Sp}(Y_z^*Y_z)}\frac{1}{|t-w|^2}\leq \frac{1}{\eta^2}
       .
     \end{equation*}
     Similarly, $G^*Y_z^*Y_z G=(Y_z^*Y_z-\xoverline{w})^{-1}Y_z^*Y_z (Y_z^*Y_z-w)^{-1}$, so that 
     \begin{equation*}
       s^2_N(Y_z G)
       \leq 
       \sup_{t\in\RR_+}\frac{t}{|t-w|^2}=\frac{\sqrt{E^2+\eta^2}+E}{2\eta^2}
       .
     \end{equation*}
     To estimate the spectral norm of the last matrix we use the identities $Y_z G^*=\Gc^* Y_z$ and $G Y_z^*=Y_z^* \Gc$, that allow us to rewrite $Y_z G Y_z^*$ as $Y_z Y_z^* \Gc$. 
     Thus,
     \begin{equation*}
       s_N^2(Y_z G Y_z^*)
       =
       s_N(\Gc^*Y_zY_z^*Y_zY_z^*\Gc)
       =
       \sup_{t\in \RR_+}\frac{t^2}{(t-E)^2+\eta^2}
       =
       1+\frac{E^2}{\eta^2}
       .
     \end{equation*}
     This implies that if $|w|\leq C$ and $\eta$ is bounded away from zero, then $\LRI{D}{3}{b}{kl} \LRI{G}{}{a}{ii}$ is bounded. The Lemma is proven.
    \end{proof}

\bibliographystyle{plain}
\bibliography{bib}
\noindent
\\
\\
\textsc{Yuriy Nemish,}
\\
\textsc{\small{Institut de Mathématiques de Toulouse,
UMR 5219 du CNRS\\
Université de Toulouse, F-31062, Toulouse, France}}
\\
\textit{E-mail address:}
\texttt{yuriy.nemish@math.univ-toulouse.fr}

\end{document}